\documentclass[11pt,reqno]{amsart}
\usepackage{amsmath,amssymb,latexsym,soul,cite,mathrsfs}
\usepackage{color,enumitem,graphicx}
\usepackage[colorlinks=true,urlcolor=blue,
citecolor=red,linkcolor=blue,linktocpage,pdfpagelabels,
bookmarksnumbered,bookmarksopen]{hyperref}
\usepackage[english]{babel}
\usepackage{enumitem}

\usepackage[left=2.7cm,right=2.7cm,top=2.9cm,bottom=2.9cm]{geometry}

\usepackage[hyperpageref]{backref}

\makeatletter
\newcommand{\leqnomode}{\tagsleft@true}
\newcommand{\reqnomode}{\tagsleft@false}
\makeatother

\numberwithin{equation}{section}

\newtheorem{theorem}{Theorem}[section]
\newtheorem{lemma}[theorem]{Lemma}

\newtheorem{proposition}[theorem]{Proposition}

\newtheorem{remark}[theorem]{Remark}

\newtheorem{theoremletter}{Theorem}

\renewcommand{\rightarrow}{\to}

\newcommand{\ud}{\mathrm{d}}

\title[Coupled systems involving the square root of the Laplacian]{Coupled elliptic systems involving the square root of the Laplacian and Trudinger-Moser critical growth}  

\author[J.M.\ do \'O]{Jo\~ao Marcos do \'O}
\author[JC. \ de Albuquerque]{Jos\'e Carlos de Albuquerque}

\address[J.M. do \'O]{Department of Mathematics,
	Federal University of Para\'{\i}ba
	\newline\indent
	58051-900, Jo\~ao Pessoa-PB, Brazil}
\email{\href{mailto:jmbo@pq.cnpq.br}{jmbo@pq.cnpq.br}}

\address[J.C. de~Albuquerque]{Department of Mathematics, Federal University of Para\'{\i}ba}
\email{\href{mailto:joserre@gmail.com}{joserre@gmail.com}}

\thanks{Corresponding author: J.M. do~\'O}

\thanks{Research supported in part by INCTmat/MCT/Brazil, CNPq and CAPES/Brazil}


\begin{document}
	
	
	\begin{abstract}
		In this paper we prove the existence of a nonnegative ground state solution to the following class of coupled systems involving Schr\"{o}dinger equations with square root of the Laplacian
		$$
		\left\{
		\begin{array}{lr}
		(-\Delta)^{1/2}u+V_{1}(x)u=f_{1}(u)+\lambda(x)v, & x\in\mathbb{R},\\
		(-\Delta)^{1/2}v+V_{2}(x)v=f_{2}(v)+\lambda(x)u, & x\in\mathbb{R},
		\end{array}
		\right.
		$$
		where the nonlinearities $f_{1}(s)$ and $f_{2}(s)$ have exponential critical growth of the Trudinger-Moser type, the potentials $V_{1}(x)$ and $V_{2}(x)$ are nonnegative and periodic. Moreover, we assume that there exists $\delta\in (0,1)$ such that $\lambda(x)\leq\delta\sqrt{V_{1}(x)V_{2}(x)}$. We are also concerned with the existence of ground states when the potentials are asymptotically periodic. Our approach is variational and based on minimization technique over the Nehari manifold.
	\end{abstract}
	\maketitle

	
	
	
	
	\section{Introduction}	
	
	This paper deals with the existence of ground states to the following class of coupled systems 
	\begin{equation}\label{paper4j0}
	\left\{
	\begin{array}{lr}
	(-\Delta)^{1/2}u+V_{1}(x)u=f_{1}(u)+\lambda(x)v, & x\in\mathbb{R},\\
	(-\Delta)^{1/2}v+V_{2}(x)v=f_{2}(v)+\lambda(x)u, & x\in\mathbb{R},
	\end{array}
	\right. \tag{$S$}
	\end{equation}
	where $(-\Delta)^{1/2}$ denotes the \textit{square root of the Laplacian}, the potentials $V_{1}(x)$, $V_{2}(x)$ are nonnegative and satisfy $\lambda(x)\leq\delta\sqrt{V_{1}(x)V_{2}(x)}$, for some $\delta\in(0,1)$ and for all $x\in\mathbb{R}$. Here we consider the case when $V_{1}(x)$, $V_{2}(x)$ and $\lambda(x)$ are periodic, and also when these functions are asymptotically periodic, that is, the limits of $V_{1}(x)$, $V_{2}(x)$ and $\lambda(x)$ are periodic functions when $|x|\rightarrow+\infty$. Our main goal here is to study the existence of ground states for \eqref{paper4j0}, involving a nonlocal operator when the nonlinearities $f_{1}(u)$, $f_{2}(v)$ have exponential critical growth motivated by a class of Trudinger-Moser inequality introduced by T.~Ozawa (see Theorem~\ref{paper4oz} in the Section~\ref{paper4s1}). 
	
	\subsection{Motivation}
	In order to motivate our results we begin by giving
	a brief survey on this subject. In the last few years, a great attention
	has been focused on the study of problems involving fractional Sobolev spaces and corresponding nonlocal equations, both from a pure mathematical point of view and their concrete applications, since they naturally arise in many different contexts, such as, among the others, obstacle problems, flame propagation, minimal surfaces, conservation laws, financial market, optimization, crystal dislocation, phase transition and water waves, see for instance \cite{guia,caffa} and references therein.
	
	Solutions of System~\eqref{paper4j0} are related with standing wave solutions of the following two-component system of nonlinear Schr\"{o}dinger equations
	\begin{equation}\label{paper4j000}
	\left\{
	\begin{array}{lr}
	i\displaystyle\frac{\partial\psi}{\partial t}=(-\Delta)^{1/2}\psi+V_{1}(x)\psi-f_{1}(\psi)-\lambda(x)\phi, & (t,x)\in\mathbb{R}\times\mathbb{R}^{N},\\
	i\displaystyle\frac{\partial\phi}{\partial t}=(-\Delta)^{1/2}\phi+V_{2}(x)\phi-f_{2}(\phi)-\lambda(x)\psi, & (t,x)\in\mathbb{R}\times\mathbb{R}^{N},
	\end{array}
	\right. 
	\end{equation}
	where $i$ denotes the imaginary unit and $N=1$. For System~\eqref{paper4j000}, a solution of the form $(\psi(x,t),\phi(x,t))=(e^{-it}u(x),e^{-it}v(x))$
	is called \textit{standing wave}. Assuming that $f_{1}(e^{i\theta}u)=e^{i\theta}f_{1}(u)$ and $f_{2}(e^{i\theta}v)=e^{i\theta}f_{2}(v)$, for $u,v\in\mathbb{R}$, it is well known that $(\psi,\phi)$ is a solution of \eqref{paper4j000} if and only if $(u,v)$ solves System~\eqref{paper4j0}. The studying of the existence of standing waves for nonlinear Schr\"{o}dinger	equations arises in various branches of mathematical physics and nonlinear topics, see \cite{fisica,bl,bl2,l1,l2,strauss} and references therein for more complete discussion of this topic.	
	
	
	It is known that when $s\rightarrow1$, the fractional Laplacian $(-\Delta)^{s}$ reduces to the standard Laplacian $-\Delta$, see \cite{guia}. Nonlinear Schr\"{o}dinger equations involving the standard Laplacian have been broadly investigated in many aspects, see for instance \cite{bl,bl2,rabinowitz} and references therein. On the nonlinear elliptic equations involving nonlinearities with critical growth of the Trudinger-Moser type, we refer the readers to \cite{cao,r1,alves,djr,djairo,severo} and references therein.
	
	
	There are some papers that have appeared in the recent years regarding the local case of System~\eqref{paper4j0}, which corresponds to the case $s=1$. For instance, in \cite{czo1}, the authors proved the existence of ground states for critical coupled systems of the form
	\begin{equation}\label{paper4chen}
	\left\{
	\begin{array}{lr}
	-\Delta u+\mu u=|u|^{p-1}u+\lambda v,     & x\in\mathbb{R}^{N},\\
	-\Delta v+\nu v=|v|^{2^{*}-1}v+\lambda u, & x\in\mathbb{R}^{N},
	\end{array}
	\right. 
	\end{equation}  
	where $0<\lambda<\sqrt{\mu\nu}$, $1<p<2^{*}-1$ and $N\geq3$. In \cite{t}, G.~Li and X.H.~Tang proved the existence of ground state for System~\eqref{paper4chen} when $\mu=a(x)$, $\nu=b(x)$ and $\lambda=\lambda(x)$ are continuous functions, $1$-periodic in each $x_{1},x_{2},...,x_{N}$ and satisfy $\lambda^{2}(x)<\sqrt{a(x)b(x)}$, for all $x\in\mathbb{R}^{N}$. For another classes of coupled systems and existence of least energy solutions, we refer the readers to \cite{acr,czo1,czo2,zhang}. Concerning nonlinear elliptic systems involving nonlinearities with critical growth of the Trudinger-Moser type, we refer the readers to \cite{djr,r2,djb,mana,lam2,rabelo} and references therein. Though there has been some works on the existence of ground states for systems involving the standard Laplacian, not much has been done for the class of nonlocal problems involving exponential critical growth.
	
	
	The fractional case, which corresponds to $0<s<1$, has been widely studied motivated by the work of L.~Caffarelli and L.~Silvestre. They proposed transform the nonlocal problem into a local problem via the Dirichlet-Neumann map, see \cite{cs,silvestre}. Recently, the fractional nonlinear Sch\"{o}dinger equation $(-\Delta)^{s}u+V(x)u=f(x,u)$ in $\mathbb{R}^{N}$, $N\geq1$, has been studied under many different assumptions on the potential $V(x)$ and on the nonlinearity $f(x,u)$. In \cite{felmer}, it was proved the existence of positive solutions for the case when $V\equiv1$ and $f(x,u)$ has subcritical growth in the Sobolev sense. In order to overcome the lack of compactness, the authors used a comparison argument. Another way to overcome this difficulty is requiring coercive potentials, that is, $V(x)\rightarrow+\infty$, as $|x|\rightarrow+\infty$. In this direction, the existence of ground states was studied by M.~Cheng, \cite{cheng}, considering a polynomial nonlinearity, and S.~Secchi, \cite{secchi}, considering a more general nonlinearity in the subcritical case. For existence results involving another types of potentials, we refer \cite{chang,feng,pala} and references therein. We point out that in all of these works it were consider dimension $N\geq2$ and nonlinearities with polynomial behavior.
	
	In the fractional case, the critical Sobolev exponent is given by $2^{*}_{s}=2N/(N-2s)$. If $0<s<N/2$, then the fractional Sobolev space $H^{s}(\mathbb{R}^{N})$ is continuously embedded into $L^{q}(\mathbb{R}^{N})$, for all $q\in[2,2^{*}_{s}]$. Thus, similarly the standard Laplacian case, the maximal growth on the nonlinearity $f(x,u)$ which allows to treat nonlinear fractional Sch\"{o}dinger equations variationally in $H^{s}(\mathbb{R}^{N})$ is given by $|u|^{2^{*}_{s}-1}$, when $|u|\rightarrow+\infty$. For $N=1$ and $s\rightsquigarrow1/2$, we have $2^{*}_{s}\rightsquigarrow+\infty$. In this case, $H^{1/2}(\mathbb{R})$ is continuously  embedded into $L^{q}(\mathbb{R})$, for all $q\in[2,+\infty)$. However, $H^{1/2}(\mathbb{R})$ is not continuously embedded into $L^{\infty}(\mathbb{R})$. For more details we refer the reader to \cite{guia} and the bibliographies therein. In this present work, we deal with the limiting case, when $N=1$, $s=1/2$ and nonlinearities with the maximum growth which allows to treat System~\eqref{paper4j0} variationally. For existence results considering the limiting case we refer the readers to \cite{yane,jms,jmm,squassina} and references therein.	
	
	
	Motivated by the above discussion, the current paper has two purposes. First, we are concerned with the existence of nonnegative ground state solution for System~\eqref{paper4j0}, for the case when $V_{1}(x)$, $V_{2}(x)$ and $\lambda(x)$ are periodic. Second, we make use of our first result to study System~\eqref{paper4j0} in the asymptotically periodic case. For that matter, we deal with several difficulties imposed by the class of systems introduced by \eqref{paper4j0}. The first one is the presence of the square root of the Laplacian which is a nonlocal operator, that is, it takes care of the behavior of the solution in the whole space. This class of systems is also characterized by its lack of compactness due to the fact that the nonlinear terms have critical growth and the equations are defined in whole Euclidean space $\mathbb{R}$, which roughly speaking, originates from the invariance of $\mathbb{R}$ with respect to translation and dilation. Furthermore, we have the fact that \eqref{paper4j0} involves strongly coupled fractional Schr\"odinger equations because of the linear terms in the right hand side. To overcome these difficulties, we shall use a variational approach based on Nehari manifold in combination with a Trudinger-Moser type inequality (see Theorem~\ref{paper4oz}) and a version of a lemma due to P.L. Lions for fractional case (see Lemma~\ref{paper4lions}). To our acknowledgment this is the first work where it is proved the existence of ground states for this class of systems under assumptions involving periodic and asymptotically periodic potentials and nonlinearities with exponential critical growth of the Trudinger-Moser type.
	
	
	\subsection{Assumptions and main results} We start this subsection recalling some preliminary concepts about the fractional operator, for a more complete discussion we cite \cite{guia}. For $s\in(0,1)$, the \textit{fractional Laplacian} of a function $u:\mathbb{R}\rightarrow\mathbb{R}$ in the Schwartz class is defined by
	$$
	(-\Delta)^{s}u=\mathcal{F}^{-1}(|\xi|^{2s}(\mathcal{F}u)), \quad \mbox{for all} \hspace{0,2cm} \xi\in\mathbb{R},
	$$
	where $\mathcal{F}$ denotes the Fourier transform
	$$
	\mathcal{F}(u)(\xi)=\frac{1}{\sqrt{2\pi}}\int_{\mathbb{R}}e^{-i\xi\cdot x}u(x)\;\ud x.
	$$
	The particular case when $s=1/2$ its called the \textit{square root of the Laplacian}. We recall the definition of the fractional Sobolev space
	$$
	H^{1/2}(\mathbb{R})=\left\{u\in L^{2}(\mathbb{R}):\int_{\mathbb{R}^{2}}\frac{|u(x)-u(y)|^{2}}{|x-y|^{2}}\;\ud x\;\ud y<\infty\right\},
	$$
	endowed with the natural norm
	$$
	\|u\|_{1/2}=\left([u]_{1/2}^{2}+\int_{\mathbb{R}}u^{2}\;\ud x\right)^{1/2}, \quad
	[u]_{1/2}=\left(\int_{\mathbb{R}^{2}}\frac{|u(x)-u(y)|^{2}}{|x-y|^{2}}\;\ud x\;\ud y\right)^{1/2}
	$$
	where the term $[u]_{1/2}$ is the so-called \textit{Gagliardo semi-norm} of the function $u$. We recall also that
	$$
	\|(-\Delta)^{1/4}u\|_{L^{2}}^{2}=\frac{1}{2\pi}\int_{\mathbb{R}^{2}}\frac{|u(x)-u(y)|^{2}}{|x-y|^{2}}\;\ud x\;\ud y \quad \mbox{for all} \hspace{0,2cm} u\in H^{1/2}(\mathbb{R}).
	$$
	In view of the potentials $V_{1}(x)$ and $V_{2}(x)$, we define the following subspace of $H^{1/2}(\mathbb{R})$
	$$
	E_{i}=\left\{u\in H^{1/2}(\mathbb{R}):\int_{\mathbb{R}}V_{i}(x)u^{2}\;\ud x<\infty\right\} \quad \mbox{for} \hspace{0,2cm} i=1,2,
	$$
	endowed with the inner product
	$$
	(u,v)=\int_{\mathbb{R}}(-\Delta)^{1/4}u(-\Delta)^{1/4}v\;\ud x+\int_{\mathbb{R}}V_{i}(x)u^{2}\;\ud x,
	$$
	to which corresponds the induced norm $\|u\|_{E_{i}}^{2}=(u,u)$. In order to establish a variational approach to treat System~\eqref{paper4j0}, we need to require suitable assumptions on the potentials. For each $i=1,2$, we assume that
	
	\begin{enumerate}[label=($V_{1}$),ref=$(V_{1})$] 
		\item \label{paper4A1}
		$V_{i}(x)$, $\lambda(x)$ are periodic, that is, $V_{i}(x)=V_{i}(x+z)$, $\lambda(x)=\lambda(x+z)$, for all $x\in\mathbb{R}$, $z\in\mathbb{Z}$.  
	\end{enumerate}
	
	\begin{enumerate}[label=($V_{2}$),ref=$(V_{2})$] 
		\item \label{paper4A2}
		$V_{i}\in L^{\infty}_{loc}(\mathbb{R})$, $V_{i}(x)\geq0$ for all $x\in\mathbb{R}$ and 
		$$
		\nu_{i}=\inf_{u\in E_{i}}\left\{\frac{1}{2\pi}[u]_{1/2}^{2}+\int_{\mathbb{R}}V_{i}(x)u^{2}\;\ud x: \int_{\mathbb{R}}u^{2}\;\ud x=1\right\}>0.
		$$			
	\end{enumerate}
	
	\begin{enumerate}[label=($V_{3}$),ref=$(V_{3})$] 
		\item \label{paper4A3}
		$\lambda(x)\leq\delta\sqrt{V_{1}(x)V_{2}(x)}$, for some $\delta\in(0,1)$, for all $x\in\mathbb{R}$.	 
	\end{enumerate}
	
	
	\noindent Using assumption \ref{paper4A2} we can see that the product space $E=E_{1}\times E_{2}$ is a Hilbert space when endowed with the scalar product
	$$
	((u,v),(w,z))=\int_{\mathbb{R}}\left((-\Delta)^{1/4}u(-\Delta)^{1/4}w+V_{1}(x)uw+(-\Delta)^{1/4}v(-\Delta)^{1/4}z+V_{2}(x)vz\right)\;\ud x,
	$$
	to which corresponds the induced norm $\|(u,v)\|_{E}^{2}=\|u\|_{E_{1}}^{2}+\|v\|_{E_{2}}^{2}$.
	
	We suppose here that the nonlinearities $f_{1}(s)$ and $f_{2}(s)$ have \textit{exponential critical growth}. Precisely, for $i=1,2$, given $\alpha_{0}^{i}>0$ we say that $f_{i}:\mathbb{R}\rightarrow\mathbb{R}$ has $\alpha_{0}^{i}$-\textit{critical growth} at $\pm\infty$~if 
	\leqnomode	
	\begin{align}\label{paper4j2}
	\limsup_{s\rightarrow\pm\infty}\frac{|f_{i}(s)|}{e^{\alpha s^{2}}-1}=
	\left\{
	\begin{array}{cll}
	0 & \mbox{if} & \alpha>\alpha^{i}_{0},\\
	+\infty & \mbox{if} & \alpha<\alpha^{i}_{0}.
	\end{array}
	\right.	
	\tag{CG}
	\end{align}
	This notion of \textit{criticality} is motivated by a class of Trudinger-Moser type inequality introduced by T.~Ozawa (see Section~\ref{paper4s1}). Furthermore, we make the following assumptions for each $i=1,2$:
	\begin{enumerate}[label=($H_1$),ref=$(H_1)$] 
		\item \label{paper4f1}
		The function $f_{i}$ belongs to $C^{1}(\mathbb{R})$, is convex on $\mathbb{R}^{+}$, $f_{i}(-s)=-f_{i}(s)$ for $s\in\mathbb{R}$, and 
		$$
		\lim_{s\rightarrow0}\frac{f_{i}(s)}{s}=0.
		$$
	\end{enumerate}
	\begin{enumerate}[label=($H_2$),ref=$(H_2)$] 
		\item \label{paper4f2}
		The function $s\mapsto s^{-1}f_{i}(s)$ is increasing for $s>0$.
	\end{enumerate}
	\begin{enumerate}[label=$(H_3)$,ref=$(H_3)$] 
		\item \label{paper4f3}
		There exists $\mu_{i}>2$ such that
		$$
		0<\mu_{i} F_{i}(s):=\mu_{i}\int_{0}^{s}f_{i}(\tau)\;\ud\tau\leq f_{i}(s)s, \hspace{0,5cm} \mbox{for all} \hspace{0,2cm} s\in\mathbb{R}\backslash\{0\}.
		$$
	\end{enumerate}
	\begin{enumerate}[label=($H_{4}$),ref=$(H_{4})$]
		\item \label{paper4f4}
		There exist $q>2$ and $\vartheta>0$ such that
		$$
		 F_{i}(s)\geq\vartheta|s|^{q}, \quad \mbox{for all} \hspace{0,2cm} s\in\mathbb{R}.
		$$
	\end{enumerate}
	
	
	
	\noindent We are in condition to state our existence theorem for the case when the potentials are periodic.
	
	\begin{theorem}\label{paper4A}
		Suppose that assumptions \ref{paper4A1}-\ref{paper4A3} hold. Assume that for each $i=1,2$ $f_{i}(s)$ and $f_{i}'(s)s$ have $\alpha_{0}^{i}$-critical growth~\eqref{paper4j2} and satisfy \ref{paper4f1}-\ref{paper4f4}. Then, System \eqref{paper4j0} possesses a nonnegative ground state solution provided $\vartheta$ in \ref{paper4f4} is large enough.
	\end{theorem}
	
	We are also concerned with the existence of ground states for the following coupled system	 
	\begin{equation}\label{paper4j00}
	\left\{
	\begin{array}{lr}
	(-\Delta)^{1/2}u+\tilde{V_{1}}(x)u=f_{1}(u)+\tilde{\lambda}(x)v, & x\in\mathbb{R},\\
	(-\Delta)^{1/2}v+\tilde{V_{2}}(x)v=f_{2}(v)+\tilde{\lambda}(x)u, & x\in\mathbb{R},
	\end{array}
	\right. \tag{$\tilde{S}$}
	\end{equation} 
	when the potentials $\tilde{V_{1}}(x),$ $\tilde{V_{2}}(x)$ and $\tilde{\lambda}(x)$ are asymptotically periodic. In analogous way, we may define the suitable space $\tilde{E}=\tilde{E_{1}}\times\tilde{E_{2}}$ considering $\tilde{V_{i}}(x)$ instead $V_{i}(x)$. In order to establish an existence theorem for \eqref{paper4j00}, for $i=1,2$ we introduce the following assumptions:
	
	\begin{enumerate}[label=($V_{4}$),ref=$(V_{4})$] 
		\item \label{paper4A4}
		$\tilde{V_{i}}(x)<V_{i}(x)$, $\lambda(x)<\tilde{\lambda}(x)$, for all $x\in \mathbb{R}$ and
		 $$
		  \lim_{|x|\rightarrow+\infty}|V_{i}(x)-\tilde{V_{i}}(x)|=0 \quad \mbox{and} \quad \lim_{|x|\rightarrow+\infty}|\tilde{\lambda}(x)-\lambda(x)|=0.
		 $$
	\end{enumerate}
	
	\begin{enumerate}[label=($V_{5}$),ref=$(V_{5})$] 
		\item \label{paper4A5}
		$\tilde{V_{i}}\in L^{\infty}_{loc}(\mathbb{R})$, $\tilde{V_{i}}(x)\geq0$ for all $x\in\mathbb{R}$ and 
		$$
		\tilde{\nu_{i}}=\inf_{u\in \tilde{E_{i}}}\left\{\frac{1}{2\pi}[u]_{1/2}^{2}+\int_{\mathbb{R}}\tilde{V_{i}}(x)u^{2}\;\ud x: \int_{\mathbb{R}}u^{2}\;\ud x=1\right\}>0.
		$$		 
	\end{enumerate}
	
	\begin{enumerate}[label=($V_{6}$),ref=$(V_{6})$] 
		\item \label{paper4A6}
		$\tilde{\lambda}(x)\leq\delta\sqrt{\tilde{V_{1}}(x)\tilde{V_{2}}(x)}$, for some $\delta\in(0,1)$, for all $x\in\mathbb{R}$.
	\end{enumerate}
	
	\begin{theorem}\label{paper4B}
		Suppose that assumptions \ref{paper4A1}-\ref{paper4A6} hold and for each $i=1,2$ assume that $f_{i}(s)$ has $\alpha_{0}^{i}$-critical growth~\eqref{paper4j2}, satisfies \ref{paper4f1}-\ref{paper4f4} and $f_{i}'(s)s$ has $\alpha_{0}^{i}$-critical growth~\eqref{paper4j2}. Then, System \eqref{paper4j00} possesses a nonnegative ground state solution provided $\vartheta$ in \ref{paper4f4} is large enough.
	\end{theorem} 
	
	
	\begin{remark} We collect the following remarks on our assumptions:
	\end{remark}
	
	\begin{enumerate}[label=\textbf{(i)},ref=(i)] 
		\item 
		A typical example of nonlinearity which satisfies the assumptions \ref{paper4f1}-\ref{paper4f4} is
		$$
		f(s)=\vartheta q|s|^{q-2}s+q|s|^{q-2}s(e^{\alpha_{0}s^{2}}-1)+2\alpha_{0}|s|^{q}se^{\alpha_{0}s^{2}}, \quad \mbox{for} \ 2<\mu<q \ \mbox{and} \ s\in\mathbb{R},
		$$
		where $\alpha_{0}$ is the critical exponent introduced in \eqref{paper4j2}.
	\end{enumerate}	
	
	\begin{enumerate}[label=\textbf{(ii)},ref=(ii)]
		\item
		The assumption \ref{paper4f4} could be replaced by the following local condition: there exists $q>2$ and $\tilde{\vartheta}$ such that
		 \begin{equation}\label{local}
		 \liminf_{s\rightarrow0}\frac{F_{i}(s)}{|s|^{q}}\geq\tilde{\vartheta}>0.
		 \end{equation}
		In fact, we can use the critical exponential growth of the nonlinearities, Ambrosetti-Rabinowitz condition \ref{paper4f3} and assumption \eqref{local} to deduce \ref{paper4f4}. In order to ease the presentation of this paper and avoid certain technicalities, we simply assume \ref{paper4f4}.		
	\end{enumerate}
	
	\begin{enumerate}[label=\textbf{(iii)},ref=(iii)] 
		\item 
		Assumption \ref{paper4f4} plays a very important role in the proof of Theorems~\ref{paper4A}~and~\ref{paper4B}. We will prove the existence of ground states when $\vartheta$ is large enough. Precisely, if
		  \begin{equation}\label{size}
		  \vartheta>\vartheta_{0}=\frac{S_{q}^{q}}{q}\left(\frac{1}{1-\delta}\frac{\mu}{\mu-2}\frac{q-2}{q}\frac{\alpha_{0}\kappa^{-1}}{\omega}\right)^{(q-2)/2},
		  \end{equation}
	    where $\alpha_{0}=\max\{\alpha_{0}^{1},\alpha_{0}^{2}\}$, $\mu=\min\{\mu_{1},\mu_{2}\}$, $\omega$ is introduced in Theorem~\ref{paper4oz}, $\kappa^{-1}=\max\{\kappa_{1}^{-1},\kappa_{2}^{-1}\}$ where $\kappa_{i}$ is introduced in Lemma~\ref{paper4emb} and $S_{q}$ is introduced in Section~\ref{paper4s3}.
	    The estimate \eqref{size} will allow us to apply the Trudinger-Moser inequality (see Section~\ref{paper4s1}, Theorem~\ref{paper4oz}) in the minimizing sequence obtained by Ekeland's variational principle (see Lemma~\ref{paper4principal}) in order to prove that the weak limit of this sequence belongs to Nehari manifold.	
	\end{enumerate}
	
	\begin{enumerate}[label=\textbf{(iv)},ref=(iv)] 
		\item 
		Theorems~\ref{paper4A}~and~\ref{paper4B} may be considered as the extension of the main result for the scalar case in \cite{jms}, because we consider a class of potentials and the nonlinear term different from them. If we take $u=v$ and $\lambda=0$ in System~\eqref{paper4j0} then we solve the single equation found in that paper but under our hypotheses.
		
	\end{enumerate}

	\subsection{Notation} We will use the following notation:
	
	\begin{itemize}
		\item $C$, $\tilde{C}$, $C_{1}$, $C_{2}$,... denote positive constants (possibly different).
		\item The norm in $L^{p}(\mathbb{R})$ and $L^{\infty}(\mathbb{R})$, will be denoted respectively by $\|\cdot\|_{L^{p}}$ and $\|\cdot\|_{L^{\infty}}$.
		\item The norm in $L^{p}(\mathbb{R})\times L^{p}(\mathbb{R})$ is given by $\|(u,v)\|_{L^{p}}=\left(\|u\|^{p}_{L^{p}}+\|v\|^{p}_{L^{p}}\right)^{1/p}$.
		\item The norm in $H^{1/2}(\mathbb{R})\times H^{1/2}(\mathbb{R})$ is given by $\|(u,v)\|_{1/2}=\left(\|u\|^{2}_{1/2}+\|v\|^{2}_{1/2}\right)^{1/2}$.
	\end{itemize}		
	
	
	\subsection{Outline} The remainder of this paper is organized as follows. In the Sections \ref{paper4s1}~and~\ref{paper4var}, we collect some results which are crucial to give a variational approach for our problem. In the Section~\ref{paper4subneh}, we introduce and give some properties of the Nehari manifold (for a more complete description of this subject, see for example \cite{neh}). In the Section \ref{paper4s3}, we study the periodic case. For this purpose, we make use of the Ekeland's variational principle to obtain a minimizing sequence for the energy functional on the Nehari manifold. We shall use a fractional version of a lemma introduced by P.L.~Lions, a Brezis-Lieb type lemma and a Trudinger-Moser type inequality to prove that the weak limit of the minimizing sequence will be a ground state solution for the problem. In the periodic case, the key point is to use the invariance of the energy functional under translations to recover the compactness of the minimizing sequence. Finally, in the Section~\ref{paper4t2} we study the asymptotically periodic case. For this matter, the key point is a relation obtained between the ground state energy associated with Systems~\eqref{paper4j0} and \eqref{paper4j00} (see Lemma~\ref{paper4est}).	
	
	
	\section{Preliminary results}\label{paper4s1}
	
	In this Section we provide preliminary results which will be used throughout the paper. One of the features of the class of the systems \eqref{paper4j0} and \eqref{paper4j00} is the presence of the nonlocal operator, square root of the Laplacian. Another feature of these classes of problems is the exponential critical behavior of the nonlinearities in the sense of Trudinger-Moser. We are motivated by the following Trudinger-Moser type inequality which was introduced by T.~Ozawa (see \cite{ozawa}).
	
	\begin{theoremletter}\label{paper4oz}
		There exists $\omega\in(0,\pi)$ such that, for all $\alpha\in(0,\omega]$, there exists $H_{\alpha}>0$ with
		\begin{equation}\label{paper4j5}
		\int_{\mathbb{R}}(e^{\alpha u^{2}}-1)\;\ud x\leq H_{\alpha}\|u\|_{L^{2}}^{2},
		\end{equation}
		for all $u\in H^{1/2}(\mathbb{R})$ such that $\|(-\Delta)^{1/4}u\|_{L^{2}}^{2}\leq 1$.
	\end{theoremletter}
	
	The following result is a consequence of Theorem~\ref{paper4oz}, more details can be found in \cite[Lemma~2.2]{jms}.
	
	\begin{lemma}\label{paper4tm}
		Let $u\in H^{1/2}(\mathbb{R})$ and $\rho_{0}>0$ be such that $\|u\|_{1/2}\leq\rho_{0}$. Then, there exists $C=C(\alpha,\rho_{0})>0$ such that
		$$
		\int_{\mathbb{R}}(e^{\alpha u^{2}}-1)\;\ud x\leq C, \quad \mbox{for every} \hspace{0,2cm} 0<\alpha\rho_{0}^{2}<\omega.
		$$ 
	\end{lemma}
	
	\begin{lemma}\label{paper4exp}
		Let $\alpha>0$ and $l>1$. Then, for each $r>l$ there exists a positive constant $C=C(r)$ such that
		$$
		(e^{\alpha s^{2}}-1)^{l}\leq C(e^{r\alpha s^{2}}-1), \quad \mbox{for all} \hspace{0,2cm} s\in\mathbb{R}.
		$$ 
	\end{lemma}
	
	\begin{remark}\label{paper4r1}
		In light of \cite[Theorem~8.5]{loss}, for any 
		$p\geq2$, there exists $C=C(p)$, such that
		\begin{equation}\label{paper4ce}
		\|u\|_{L^{p}}\leq C\|u\|_{1/2}, \quad \mbox{for all} \hspace{0,2cm} u\in H^{1/2}(\mathbb{R}).
		\end{equation}
	\end{remark}
	
	\begin{lemma}\label{paper4emb}
		Assume that \ref{paper4A2} holds. Then for each $i=1,2$ there exists $\kappa_{i}>0$ such that 
		\begin{equation}\label{paper4jjj1} 
		\kappa_{i}\|u\|_{1/2}^{2}\leq \frac{1}{2\pi}[u]_{1/2}^{2}+\int_{\mathbb{R}}V_{i}(x)u^{2}\;\ud x, \quad \mbox{for all} \hspace{0,2cm} u\in E_{i}.
		\end{equation}
	\end{lemma}
	\begin{proof}
		Suppose that \eqref{paper4jjj1} does not holds. Thus, there exists a sequence $(u_{n})_{n}\subset E_{i}$ such that $\|u_{n}\|_{1/2}=1$ and
		$$
		\frac{1}{2\pi}[u_{n}]_{1/2}^{2}+\int_{\mathbb{R}}V_{i}(x)u_{n}^{2}\;\ud x<\frac{1}{n}.
		$$
		By using \ref{paper4A2}, we have that
		$$
		0<\lambda_{i}\leq\frac{1}{\|u_{n}\|_{L^{2}}^{2}}\left(\frac{1}{2\pi}[u_{n}]_{1/2}^{2}+\int_{\mathbb{R}}V_{i}(x)u_{n}^{2}\;\ud x\right)<\frac{1}{n}\frac{1}{\|u_{n}\|_{L^{2}}^{2}},
		$$
		which implies that $\|u_{n}\|_{L^{2}}^{2}\rightarrow0$ and $[u_{n}]_{1/2}^{2}\rightarrow1$. Therefore, since $V_{i}\geq0$, we conclude that
		$$
		o_{n}(1)=-\|u_{n}\|_{L^{2}}^{2}\leq\int_{\mathbb{R}}V_{i}(x)u_{n}^{2}\;\ud x<\frac{1}{n}-\frac{1}{2\pi}[u_{n}]_{1/2}^{2}\rightarrow-\frac{1}{2\pi},
		$$
		which is impossible and finishes the proof.
	\end{proof}
	
	\noindent Notice that combining Remark~\ref{paper4r1} and Lemma~\ref{paper4emb} we have that $E_{i}$ is continuously embedded into $L^{p}(\mathbb{R})$, for any $p\geq2$. As consequence of the assumption \ref{paper4A3} we have the following lemma:
	
	\begin{lemma}\label{paper4nehari-1}
		For any $(u,v)\in E$ we have
		\begin{equation}\label{paper4j27}
		\|(u,v)\|_{E}^{2}-2\int_{\mathbb{R}}\lambda(x)uv\;\ud x\geq(1-\delta)\|(u,v)\|_{E}^{2}.
		\end{equation} 
	\end{lemma}
	\begin{proof}
		Notice that for any $(u,v)\in E$ we have that
		$$
		0\leq\left(\sqrt{V_{1}(x)}|u|-\sqrt{V_{2}(x)}|v|\right)^{2}=V_{1}(x)u^{2}-2\sqrt{V_{1}(x)}|u|\sqrt{V_{2}(x)}|v|+V_{2}(x)v^{2},
		$$
		which together with assumption \ref{paper4A3} implies that
		$$
		-2\int_{\mathbb{R}}\lambda(x)uv\;\ud x\geq-\delta\left(\int_{\mathbb{R}}V_{1}(x)u^{2}\;\ud x+\int_{\mathbb{R}}V_{2}(x)v^{2}\;\ud x\right)\geq-\delta\|(u,v)\|_{E}^{2},
		$$	 
		which implies that \eqref{paper4j27} holds.
	\end{proof}	
	
	The next lemma is a very important tool to overcome the lack of compactness. The \textit{vanishing lemma} was proved originally by P.L.~Lions \cite[Lemma~I.1]{lionss} and here we use the following version to fractional Sobolev spaces.
	
	\begin{lemma}\label{paper4lions}
		Assume that $(u_{n})_{n}$ is a bounded sequence in $H^{1/2}(\mathbb{R})$ satisfying
		\begin{align}\label{paper4lionsa} 
		\lim_{n\rightarrow+\infty}\sup_{y\in\mathbb{R}}\int_{y-R}^{y+R}|u_{n}|^{2}\;\ud x=0,
		\end{align}
		for some $R>0$. Then, $u_{n}\rightarrow0$ strongly in $L^{p}(\mathbb{R})$, for $2<p<\infty$.
	\end{lemma}
	\begin{proof}
		Given $r>p$, $R>0$ and $y\in\mathbb{R}$ it follows by standard interpolation that
		$$
		\|u_{n}\|_{L^{p}(B_{R}(y))}\leq\|u_{n}\|_{L^{2}(B_{r}(y))}^{1-\theta}\|u_{n}\|_{L^{r}(B_{R}(y))}^{\theta},
		$$
		for some $\theta\in(0,1)$ such that
		$$
		\frac{1-\theta}{2}+\frac{\theta}{r}=\frac{1}{q}.
		$$
		Using a locally finite covering of $\mathbb{R}$ consisting of open balls of radius $R$, the continuous embedding $H^{1/2}(\mathbb{R})\hookrightarrow L^{r}(\mathbb{R})$, the fact that $\|u_{n}\|_{1/2}\leq C$ and assumption \eqref{paper4lionsa}, we can conclude that
		$$
		\lim_{n\rightarrow+\infty}\|u_{n}\|_{L^{p}}\leq C\lim_{n\rightarrow+\infty}\sup_{y\in\mathbb{R}}\int_{y-R}^{y+R}|u_{n}|^{2}\;\ud x=0.
		$$
	\end{proof}
	
	
	\section{The Variational Setting}\label{paper4var}
	
	Associated to System~\eqref{paper4j0} we consider the energy functional $I:E\rightarrow\mathbb{R}$ defined by
	$$
	I(u,v)=\frac{1}{2}\left(\|(u,v)\|_{E}^{2}-2\int_{\mathbb{R}}\lambda(x)uv\;\ud x\right)-\int_{\mathbb{R}}\left(F_{1}(u)+F_{2}(v)\right)\;\ud x.
	$$
	Under our assumptions on $f_{i}(s)$, $V_{i}(x)$ and $\lambda(x)$, its standard to check that $I$ is well defined. Moreover, $I\in C^{2}(E,\mathbb{R})$ and its differential is given by
	$$
	\langle I'(u,v),(\phi,\psi)\rangle= ((u,v),(\phi,\psi))-\int_{\mathbb{R}}\left(f_{1}(u)\phi+f_{2}(v)\psi\right)\;\ud x-\int_{\mathbb{R}}\lambda(x)\left(u\psi+v\phi\right)\;\ud x.
	$$
	The critical points of $I$ are precisely solutions (in the weak sense) to \eqref{paper4j0}. We say that a solution $(u_{0},v_{0})\in E$ of \eqref{paper4j0} is a \textit{ground state solution} (or \textit{least energy solution}) if $(u_{0},v_{0})\neq(0,0)$ and its energy is minimal among the energy of all nontrivial solutions, that is, $I(u_{0},v_{0})\leq I(u,v)$ for any nontrivial solution $(u,v)\in E$ of \eqref{paper4j0}.
	
	\begin{remark}
		By using \ref{paper4f1}-\ref{paper4f3}, the following facts can be deduced for $i=1,2$ and $s\in\mathbb{R}\backslash\{0\}$:
		\reqnomode
		\begin{align}
		\label{paper4j12}
		f_{i}'(s)s^{2}-f_{i}(s)s>0,\\		
		\label{paper4j13}
		f_{i}'(s)>0,\\	
		\label{paper4j14}
		\phi_{i}(s)=f_{i}(s)s-2F_{i}(s)>0,\\		
		\label{paper4j31}
		\phi_{i}(s)>\phi_{i}(ts), \ \mbox{for all} \ t\in(0,1).	
		\end{align}
	\end{remark}
	
	\begin{lemma}
		Suppose that \ref{paper4f1} and \ref{paper4f3} hold. If $f_{i}(s)$ and $f_{i}'(s)s$ have $\alpha_{0}^{i}$-critical growth, then for each $i=1,2$, for any $\varepsilon>0$, $\alpha>\alpha_{0}$ and $p>2$, there exists $C=C(\varepsilon,p)>0$ such that
		\reqnomode
		\begin{align}
		\label{paper4j15}
		f_{i}(s)\leq\varepsilon |s|+C(e^{\alpha s^{2}}-1)|s|^{p-1},\\
		\label{paper4j39}
		f'_{i}(s)s\leq\varepsilon |s|+C(e^{\alpha s^{2}}-1)|s|^{p-1},\\
		\label{paper4j28}
		F_{i}(s)\leq\varepsilon s^{2}+C(e^{\alpha s^{2}}-1)|s|^{p}.
		\end{align}
	\end{lemma}
	
	\begin{proof}
		Let $\varepsilon>0$ be fixed. By using \ref{paper4f1}, there exists $\delta>0$ such that
		\begin{equation}\label{paper4j16}
		f_{i}(s)s\leq \varepsilon s^{2}, \quad \mbox{for all}  \hspace{0,2cm} |s|<\delta.
		\end{equation} 
		By using \eqref{paper4j2} for $\alpha>\alpha^{i}_{0}$, there exists $R>0$ such that
		\begin{equation}\label{paper4j17}
		f_{i}(s)\leq \varepsilon (e^{\alpha s^{2}}-1)\leq C(\varepsilon,p)(e^{\alpha s^{2}}-1)|s|^{p-1}, \quad \mbox{for all}  \hspace{0,2cm} |s|>R.
		\end{equation}
		By continuity we have
		\begin{equation}\label{paper4j18}
		f_{i}(s)\leq C(\varepsilon,p)(e^{\alpha s^{2}}-1)|s|^{p-1}, \quad \mbox{for all} \hspace{0,2cm} s\in[\delta,R].
		\end{equation} 
		Combining \eqref{paper4j16}, \eqref{paper4j17} and \eqref{paper4j18} we get \eqref{paper4j15}. In analogous way we get \eqref{paper4j39}. The last estimate follows from \ref{paper4f3} and \eqref{paper4j15}.
	\end{proof}
	
	
	\section{The Nehari manifold}\label{paper4subneh}
	
	In order to prove the existence of ground state for System \eqref{paper4j0}, we define the Nehari manifold
	$$
	\mathcal{N}=\left\{(u,v)\in E\backslash\{(0,0)\}:\langle I'(u,v),(u,v)\rangle=0\right\}.
	$$
	Notice that if $(u,v)\in\mathcal{N}$ then
	\begin{equation}\label{paper4j8}
	\|(u,v)\|_{E}^{2}-2\int_{\mathbb{R}}\lambda(x)uv\;\ud x=\int_{\mathbb{R}}f_{1}(u)u\;\ud x+\int_{\mathbb{R}}f_{2}(v)v\;\ud x.
	\end{equation}
	
	\begin{lemma}\label{paper4nehari}
		$\mathcal{N}$ is a $C^{1}$-manifold and there exists $\rho>0$, such that
		\begin{equation}\label{paper4j7}
		\|(u,v)\|_{E}\geq\rho, \quad \mbox{for all} \hspace{0,2cm} (u,v)\in\mathcal{N}.
		\end{equation}
	\end{lemma}
	\begin{proof}
		Let $J:E\backslash\{(0,0)\}\rightarrow\mathbb{R}$ be the $C^{1}$-functional defined by
		$$
		J(u,v)=\langle I'(u,v),(u,v)\rangle=\|(u,v)\|_{E}^{2}-2\int_{\mathbb{R}}\lambda(x)uv\;\ud x-\int_{\mathbb{R}}f_{1}(u)u\;\ud x-\int_{\mathbb{R}}f_{2}(v)v\;\ud x.
		$$
		Notice that $\mathcal{N}=J^{-1}(0)$. If $(u,v)\in\mathcal{N}$, it follows from \eqref{paper4j12} and \eqref{paper4j8} that 
		\begin{equation}\label{paper4j41}
		\langle J'(u,v),(u,v)\rangle =\int_{\mathbb{R}}\left(f_{1}(u)u-f_{1}'(u)u^{2}\right)\;\ud x+\int_{\mathbb{R}}\left(f_{2}(v)v-f_{2}'(v)v^{2}\right)\;\ud x<0.
		\end{equation}
		Therefore, $0$ is a regular value of $J$ which implies that $\mathcal{N}$ is a $C^{1}$-manifold.
		
		To prove the second part, we suppose by contradiction that \eqref{paper4j7} does not hold. Thus, we have a sequence 
		\begin{equation}\label{paper4contr}
		(u_{n},v_{n})_{n}\subset\mathcal{N}, \hspace{0,2cm} \mbox{such that} \hspace{0,2cm} \|(u_{n},v_{n})\|_{E}\rightarrow0 \hspace{0,2cm} \mbox{as} \hspace{0,2cm} n\rightarrow+\infty.
		\end{equation}
		Consider $\alpha>\alpha_{0}$ and $\rho_{0}>0$ such that $\alpha\rho_{0}^{2}<\omega$. As consequence of \eqref{paper4contr}, there exists $n_{0}\in\mathbb{N}$ such that $\kappa^{-1}\|(u_{n},v_{n})\|_{E}^{2}\leq\rho_{1}^{2}<\rho_{0}^{2}$, for $n\geq n_{0}$, where $\kappa^{-1}=\max\{\kappa_{1}^{-1},\kappa_{2}^{-1}\}$. For given $p>2$ and $\varepsilon>0$, it follows from estimate \eqref{paper4j15} that
		\begin{equation}\label{paper4j111}
		\int_{\mathbb{R}}f_{1}(u_{n})u_{n}\;\ud x\leq \varepsilon\|u_{n}\|_{L^{2}}^{2}+C_{2}\int_{\mathbb{R}}(e^{\alpha u_{n}^{2}}-1)|u_{n}|^{p}\;\ud x.
		\end{equation}
		Let $r>l>1$ be sufficiently close to $1$ such that $r\alpha\rho_{0}^{2}<\omega$. Thus, it follows from Lemma~\ref{paper4tm}, Lemma~\ref{paper4exp} and H\"older inequality that
		$$
		\int_{\mathbb{R}}(e^{\alpha u_{n}^{2}}-1)|u_{n}|^{p}\;\ud x\leq \left(\int_{\mathbb{R}}(e^{r\alpha u_{n}^{2}}-1)\;\ud x\right)^{1/l}\|u_{n}\|_{L^{pl'}}^{p}\leq C\|u_{n}\|_{L^{pl'}}^{p},
		$$
		which together with \eqref{paper4j111} and Sobolev embedding implies that
		$$
		\int_{\mathbb{R}}f_{1}(u_{n})u_{n}\;\ud x\leq \varepsilon C_{1}\|u_{n}\|_{E_{1}}^{2}+C_{2}\|u_{n}\|_{E_{1}}^{p}\leq \varepsilon C_{1}\|(u_{n},v_{n})\|_{E}^{2}+C_{2}\|(u_{n},v_{n})\|_{E}^{p}.
		$$
		Analogously, we deduce that
		$$
		\int_{\mathbb{R}}f_{2}(v_{n})v_{n}\;\ud x\leq\varepsilon C_{1}\|(u_{n},v_{n})\|_{E}^{2}+C_{2}\|(u_{n},v_{n})\|_{E}^{p}.
		$$
		Combining theses estimates we get,
		\begin{equation}\label{paper4j20}
		\int_{\mathbb{R}}(f_{1}(u_{n})u_{n}+f_{2}(v_{n})v_{n})\;\ud x\leq\varepsilon C_{1}\|(u_{n},v_{n})\|_{E}^{2}+C_{2}\|(u_{n},v_{n})\|_{E}^{p}.
		\end{equation}
		Since $\varepsilon>0$ is arbitrary and $C_{1}$ does not depend of $\varepsilon$ and $n$, we can choose $\varepsilon$ sufficiently small such that $1-\delta-\varepsilon C_{1}>0$. Thus, combining \eqref{paper4j27}, \eqref{paper4j20} and the fact that $(u_{n},v_{n})_{n}\subset\mathcal{N}$ we get
		$$
		(1-\delta)\|(u_{n},v_{n})\|_{E}^{2} \leq \int_{\mathbb{R}}(f_{1}(u_{n})u_{n}+f_{2}(v_{n})v_{n})\;\ud x \leq\varepsilon C_{1}\|(u_{n},v_{n})\|_{E}^{2}+C_{2}\|(u_{n},v_{n})\|_{E}^{p},
		$$
		which yields
		$$
		0<(1-\delta-\varepsilon C_{1})\|(u_{n},v_{n})\|_{E}^{2}\leq C_{2}\|(u_{n},v_{n})\|_{E}^{p}.
		$$
		Hence, denoting  $\rho_{2}=(1-\delta-\varepsilon C_{1})/C_{2}$ we obtain
		$$
		0<\rho_{2}^{1/(p-2)}\leq\|(u_{n},v_{n})\|_{E}.
		$$
		Choosing $\rho_{1}<\rho=\min\{\rho_{0},\rho_{2}^{1/(p-2)}\}$ we get a contradiction and we conclude that \eqref{paper4j7} holds.
	\end{proof}
	
	\begin{remark}\label{paper4remark}
		If $(u_{0},v_{0})\in\mathcal{N}$ is a critical point of $I\mid_{\mathcal{N}}$, then $I'(u_{0},v_{0})=0$. In fact, recall the notation $J(u_{0},v_{0})=\langle I'(u_{0},v_{0}),(u_{0},v_{0})\rangle$ and notice that 
		$
		I'(u_{0},v_{0})=\eta J'(u_{0},v_{0}),
		$
		where $\eta\in\mathbb{R}$ is the corresponding Lagrange multiplier. Taking the scalar product with $(u_{0},v_{0})$ and using \eqref{paper4j41} we conclude that $\eta=0$.
	\end{remark}
	
	Let us define the ground state energy associated with System~\eqref{paper4j0}, that is, $c_{\mathcal{N}}=\inf_{\mathcal{N}} I(u,v)$.
	We claim that $c_{\mathcal{N}}$ is positive. In fact, if $(u,v)\in\mathcal{N}$ it follows from \ref{paper4f3} that
	\begin{eqnarray*}
		I(u,v)	& \geq & \frac{1}{2}\left(\|(u,v)\|_{E}^{2}-2\int_{\mathbb{R}}\lambda(x)uv\;\ud x\right)-\frac{1}{\mu_{1}}\int_{\mathbb{R}}f_{1}(u)u\;\ud x-\frac{1}{\mu_{2}}\int_{\mathbb{R}}f_{2}(v)v\;\ud x\\
		& \geq & \left(\frac{1}{2}-\frac{1}{\mu}\right)\left(\|(u,v)\|_{E}^{2}-2\int_{\mathbb{R}}\lambda(x)uv\;\ud x\right),
	\end{eqnarray*}
	which together with \eqref{paper4j27} implies that
	$$   I(u,v)\geq\left(\frac{1}{2}-\frac{1}{\mu}\right)(1-\delta)\|(u,v)\|_{E}^{2}\geq \left(\frac{1}{2}-\frac{1}{\mu}\right)(1-\delta)\rho>0.
	$$
	
	
	\begin{lemma}\label{paper4j23}
		Suppose that \ref{paper4A3} and \ref{paper4f1}-\ref{paper4f4} hold. For any $(u,v)\in E\backslash\{(0,0)\}$, there exists a unique $t_{0}>0$, depending only of $(u,v)$, such that
		$$
		(t_{0}u,t_{0}v)\in\mathcal{N} \quad \mbox{and} \quad  I(t_{0}u,t_{0}v)=\max_{t\geq0} I(tu,tv).
		$$
		Moreover, if $\langle I'(u,v),(u,v)\rangle<0$, then $t_{0}\in(0,1)$.
	\end{lemma}
	\begin{proof}
		Let $(u,v)\in E\backslash\{(0,0)\}$ be fixed and consider the function $g:[0,\infty)\rightarrow\mathbb{R}$ defined by $g(t)= I(tu,tv)$. Notice that
		$$
		\langle I'(tu,tv),(tu,tv)\rangle=tg'(t).
		$$
		The result follows if we find a positive critical point of $g$. After integrating \ref{paper4f3}, we deduce that
		$$
		F_{i}(s)\geq C_{0}(|s|^{\mu_{i}}-1), \quad \mbox{for all} \hspace{0,2cm} s\neq0,
		$$
		which jointly with Lemma~\ref{paper4nehari-1} implies that
		$$
		g(t)\leq \frac{t^{2}}{2}\left(\|(u,v)\|_{E}^{2}-2\int_{\mathbb{R}}\lambda(x)uv\;\ud x\right)-C_{0}\int_{-R}^{R}(t^{\mu_{1}}|u|^{\mu_{1}}+t^{\mu_{2}}|v|^{\mu_{2}})\;\ud x-\tilde{C}.
		$$
		Since $\mu_{1}, \ \mu_{2}>2$, we obtain $g(t)<0$ for $t>0$ large. On the other hand, for some $\alpha>\alpha_{0}$ and $\rho_{0}>0$ satisfying $\alpha\rho_{0}^{2}<\omega$, we consider $t>0$ sufficiently small such that $t\kappa^{-1}\|(u,v)\|_{E}^{2}<\rho_{0}^{2}$. Thus, for $\varepsilon>0$ and $p>2$, we can use \eqref{paper4j28} and the same ideas used to obtain \eqref{paper4j20} to get  
		\begin{equation}\label{paper4j36}
		\int_{\mathbb{R}}(F_{1}(tu)+F_{2}(tv))\;\ud x\leq\varepsilon C_{1}\frac{t^{2}}{2}\|(u,v)\|_{E}^{2}+C_{2}t^{p}\|(u,v)\|_{E}^{p}.
		\end{equation}
		Since $C_{1}$ does not depends of $\varepsilon$ which is arbitrary, we can take it small enough such that $1-\delta-C_{1}\varepsilon>0$. Hence, by using \eqref{paper4j27} and \eqref{paper4j36} we have
		$$
		g(t) \geq t^{2}\|(u,v)\|_{E}^{2}\left(\frac{1-\delta-C_{1}}{2}-C_{2}t^{p-2}\|(u,v)\|_{E}^{p-2}\right).
		$$
		Thus, $g(t)>0$ provided $t>0$ is sufficiently small. Therefore, $g$ has maximum points in $(0,\infty)$. In order to prove the uniqueness, we note that every critical point of $g$ satisfies
		\begin{equation}\label{paper4j29}
		\|(u,v)\|_{E}^{2}-2\int_{\mathbb{R}}\lambda(x)uv\;\ud x=\int_{\mathbb{R}}\frac{f_{1}(tu)u}{t}\;\ud x+\int_{\mathbb{R}}\frac{f_{2}(tv)v}{t}\;\ud x.
		\end{equation}
		Furthermore, by using \eqref{paper4j12} we get
		\begin{equation}\label{paper4jj23} 
		\frac{d}{dt}\left(\frac{f_{i}(ts)s}{t}\right)=\frac{f_{i}'(ts)ts^{2}-f_{i}(ts)s}{t^{2}}=\frac{f_{i}'(ts)t^{2}s^{2}-f_{i}(ts)ts}{t^{3}}>0,
		\end{equation} 
		which implies that the right-hand side of \eqref{paper4j29} is strictly increasing on $t>0$, and consequently, the critical point $t_{0}\in(0,+\infty)$ is \textit{unique}. Finally, we assume that $\langle I'(u,v),(u,v)\rangle<0$ and we suppose by contradiction that $t_{0}\geq1$. Since $t_{0}$ is a critical point of $g$, we have
		$$
		0=g'(t_{0})=\|(u,v)\|_{E}^{2}-2\int_{\mathbb{R}}\lambda(x)uv\;\ud x-\int_{\mathbb{R}}\frac{f_{1}(t_{0}u)u}{t_{0}}\;\ud x+\int_{\mathbb{R}}\frac{f_{2}(t_{0}v)v}{t_{0}}\;\ud x.
		$$
		Therefore, by using the monotonicity obtained above, we conclude that
		$$
		0\leq\|(u,v)\|_{E}^{2}-2\int_{\mathbb{R}}\lambda(x)uv\;\ud x-\int_{\mathbb{R}}f_{1}(u)u\;\ud x+\int_{\mathbb{R}}f_{2}(v)v\;\ud x=\langle I'(u,v),(u,v)\rangle<0,
		$$
		which is a contradiction and the lemma is proved.
	\end{proof}
	
	
	\section{Proof of Theorem \ref{paper4A}}\label{paper4s3}
	
	For $q>2$ considered in \ref{paper4f4}, we define the constant
	$$
	S_{q}=\inf_{(u,v)\in E\backslash\{(0,0)\}}S_{q}(u,v),
	$$
	where
	$$
	S_{q}(u,v)=\frac{\displaystyle\left(\|(u,v)\|_{E}^{2}-2\int_{\mathbb{R}}\lambda(x)uv\;\ud x\right)^{1/2}}{\|(u,v)\|_{L^{q}}}, \quad \mbox{for} \hspace{0,2cm} (u,v)\in E\backslash\{(0,0)\}.
	$$
	
	\begin{lemma}\label{paper4sq}
		Let $\vartheta$ and $q$ be the constants introduced in \ref{paper4f4}. 
		\begin{enumerate}[label=(a),ref=(a)] 
			\item \label{paper4aa}
			The constant $S_{q}$ is positive.
		\end{enumerate}
		\begin{enumerate}[label=(b),ref=(b)] 
			\item \label{paper4bb}
			For any $(u,v)\in E\backslash\{(0,0)\}$, we have
			$$
			\max_{t\geq0}\left(\frac{t^{2}}{2}S_{q}(u,v)^{2}\|(u,v)\|_{L^{q}}^{2}-\vartheta t^{q}\|(u,v)\|_{L^{q}}^{q}\right)=\left(\frac{1}{2}-\frac{1}{q}\right)\frac{S_{q}(u,v)^{2q/(q-2)}}{(q\vartheta)^{2/(q-2)}}.
			$$
		\end{enumerate}	
	\end{lemma}
	\begin{proof}
		It follows from \eqref{paper4ce} and \eqref{paper4j27} that 
		$$
		\|(u,v)\|_{E}^{2}-\int_{\mathbb{R}}\lambda(x)uv\;\ud x\geq(1-\delta)\|(u,v)\|_{E}^{2}\geq(1-\delta)C^{-1}\|(u,v)\|_{L^{q}}^{2},
		$$
		for all $(u,v)\in E\backslash\{(0,0)\}$. Therefore, $S_{q}\geq \tilde{C}>0$.
		
		Concerning \ref{paper4bb}, for any $(u,v)\in E\backslash\{(0,0)\}$ we consider $h:[0,+\infty)\rightarrow\mathbb{R}$ defined by
		$$
		h(t)=\frac{t^{2}}{2}S_{q}(u,v)^{2}\|(u,v)\|_{L^{q}}^{2}-\vartheta t^{q}\|(u,v)\|_{L^{q}}^{q}.
		$$
		Since
		$$
		h'(t)=tS_{q}(u,v)^{2}\|(u,v)\|_{L^{q}}^{2}-q\vartheta t^{q-1}\|(u,v)\|_{L^{q}}^{q},
		$$
		it is easy to see that $h'(t)\geq0$ if and only if
		$$
		t\leq\left(\frac{S_{q}(u,v)^{2}}{q\vartheta\|(u,v)\|_{L^{q}}^{q-2}}\right)^{1/(q-2)}=\overline{t}.
		$$
		Therefore, $\overline{t}$ is a maximum point for $h$ and
		$$
		\max_{t\geq0}h(t)=h(\overline{t})=\left(\frac{1}{2}-\frac{1}{q}\right)\frac{S_{q}(u,v)^{2q/(q-2)}}{(q\vartheta)^{2/(q-2)}},
		$$
		which finishes the proof.
	\end{proof}
	
	By Ekeland's variational principle (see \cite{ekeland}), there exists a sequence $(u_{n},v_{n})_{n}\subset\mathcal{N}$ such that
	\begin{equation}\label{paper4j21}
	I(u_{n},v_{n})\rightarrow c_{\mathcal{N}} \quad \mbox{and} \quad  I'\mid_{\mathcal{N}}(u_{n},v_{n})\rightarrow0.
	\end{equation}
	Now we summarize some properties of $(u_{n},v_{n})_{n}$ which are useful to study our problem.
	\begin{lemma}\label{paper4principal}
		The minimizing sequence $(u_{n},v_{n})_{n}$ satisfies the following properties:
		\begin{enumerate}[label=(a),ref=(a)] 
			\item \label{paper4a}
			$(u_{n},v_{n})_{n}$ is bounded in $E$.
		\end{enumerate}
		\begin{enumerate}[label=(b),ref=(b)] 
			\item \label{paper4c}
			$
			\displaystyle\limsup_{n\rightarrow+\infty}\|(u_{n},v_{n})\|_{E}^{2}\leq\frac{1}{1-\delta}\frac{\mu}{\mu-2}\frac{q-2}{q}\frac{S_{q}^{2q/(q-2)}}{(q\vartheta)^{2/(q-2)}}.
			$
		\end{enumerate}
		\begin{enumerate}[label=(c),ref=(c)] 
			\item \label{paper4d}
			$(u_{n},v_{n})_{n}$ does not converge strongly to zero in $L^{m}(\mathbb{R})\times L^{m}(\mathbb{R})$, for some $m>2$.
		\end{enumerate}
		\begin{enumerate}[label=(d),ref=(d)] 
			\item \label{paper4e}
			There exists a sequence $(y_{n})_{n}\subset\mathbb{R}$ and constants $\beta,R>0$ such that
			\begin{equation}\label{paper4af1}
			\liminf_{n\rightarrow+\infty}\int_{y_{n}-R}^{y_{n}+R}(u_{n}^{2}+v_{n}^{2})\;\ud x\geq\beta>0.
			\end{equation}
		\end{enumerate}
	\end{lemma}
	\begin{proof}
		It follows from assumption \eqref{paper4j21} that
		$$
		c_{\mathcal{N}} +   o_{n}(1)=I(u_{n},v_{n})=\frac{1}{2}\left(\|(u_{n},v_{n})\|_{E}^{2}-2\int_{\mathbb{R}}\lambda(x)u_{n}v_{n}\;\ud x\right)-\int_{\mathbb{R}}(F_{1}(u_{n})+F_{2}(v_{n}))\;\ud x.
		$$
		Thus, by using \ref{paper4f3}, \eqref{paper4j27} and the fact that $(u_{n},v_{n})_{n}\subset\mathcal{N}$, we deduce that
		$$
		c_{\mathcal{N}} + o_{n}(1)\geq\left(\frac{1}{2}-\frac{1}{\mu}\right)(1-\delta)\|(u_{n},v_{n})\|_{E}^{2}.
		$$
		Therefore, $(u_{n},v_{n})_{n}$ is bounded in $E$. Moreover, the preceding estimate also implies that
		\begin{equation}\label{paper4j34}
		\limsup_{n\rightarrow\infty}\|(u_{n},v_{n})\|_{E}^{2}\leq\frac{1}{1-\delta}\frac{2\mu}{\mu-2}c_{\mathcal{N}}.
		\end{equation}
		
		To prove item~(b), we have from \ref{paper4f4} that
		\begin{equation}\label{paper4jjj2}
		 F_{1}(s)+F_{2}(t)\geq\vartheta(|s|^{q}+|t|^{q}), \quad \mbox{for all} \hspace{0,2cm} s,t\in\mathbb{R}.
		\end{equation}
		By using Lemma \ref{paper4j23}, for any $(\psi,\phi)\in E\backslash\{(0,0)\}$ there exists a unique $t_{0}>0$ such that $(t_{0}\psi,t_{0}\phi)\in\mathcal{N}$. Thus, since that $c_{\mathcal{N}}\leq I(t_{0}\psi,t_{0}\phi)\leq \max_{t\geq0} I(t\psi,t\phi)$, we can use \eqref{paper4jjj2} to~get
		$$
		c_{\mathcal{N}} \leq \max_{t\geq0}\left\{\frac{t^{2}}{2}\left(\|(\psi,\phi)\|^{2}-2\int_{\mathbb{R}}\lambda(x)\psi\phi\;\ud x\right)-\vartheta t^{q}\|(\psi,\phi)\|_{L^{q}}^{q}\right\}.
		$$
		Recalling the definition of $S_{q}(\psi,\phi)$ and using Lemma~\ref{paper4sq}~\ref{paper4bb}, we conclude that
		\begin{equation}\label{paper4j35}
		c_{\mathcal{N}}\leq\max_{t\geq0}\left\{\frac{t^{2}}{2}S_{q}(\psi,\phi)^{2}\|(\psi,\phi)\|_{L^{q}}^{2}-\vartheta t^{q}\|(\psi,\phi)\|_{L^{q}}^{q}\right\}=\left(\frac{1}{2}-\frac{1}{q}\right)\frac{S_{q}(\psi,\phi)^{2q/(q-2)}}{(q\vartheta)^{2/(q-2)}}.
		\end{equation}
		Combining \eqref{paper4j34}, \eqref{paper4j35} and taking the infimum over $(\psi,\phi)\in E\backslash\{(0,0)\}$ we have that
		$$
		\limsup_{n\rightarrow\infty}\|(u_{n},v_{n})\|_{E}^{2}\leq\frac{1}{1-\delta}\frac{\mu}{\mu-2}\frac{q-2}{q}\frac{S_{q}^{2q/(q-2)}}{(q\vartheta)^{2/(q-2)}}.
		$$


		Concerning \ref{paper4d}, let $\alpha,\rho_{0}>0$ be such that $\alpha>\alpha_{0}$ and $0<\alpha\rho_{0}^{2}<\omega$. By using item~\ref{paper4c}, there exists $\vartheta_{0}>0$ such that
		$$
		\kappa^{-1}\limsup_{n\rightarrow+\infty}\|(u_{n},v_{n})\|_{E}^{2}\leq\rho_{0}^{2}, \quad \mbox{for} \hspace{0,2cm} \vartheta>\vartheta_{0}.
		$$
		By similar arguments used in the proof of Lemma~\ref{paper4nehari}, for given $p>2$, $r>l>1$, sufficiently close to $1$, such that $r\alpha\rho_{0}^{2}<\omega$ and a suitable $\varepsilon>0$, we can deduce that
		$$
		0<(1-\delta-\varepsilon C_{1})\rho^{2}\leq (1-\delta-\varepsilon C_{1})\|(u_{n},v_{n})\|_{E}^{2}\leq C_{2}\|(u_{n},v_{n})\|_{L^{pl'}(\mathbb{R})}^{p},
		$$
		where $1/l+1/l'=1$. Therefore, $(u_{n},v_{n})_{n}$ cannot converge to zero in $L^{pl'}(\mathbb{R})$. 
		
		Finally to prove item~\ref{paper4e} we suppose by contradiction that \eqref{paper4af1} does not holds. Thus, for any $R>0$, we have
		$$
		\lim_{n\rightarrow+\infty}\sup_{y\in\mathbb{R}}\int_{y-R}^{y+R}(u_{n}^{2}+v_{n}^{2})\;\ud x=0.
		$$
		By using Lemma~\ref{paper4lions}, it follows that $(u_{n},v_{n})\rightarrow0$ strongly in $L^{p}(\mathbb{R})\times L^{p}(\mathbb{R}^{2})$ for any $p>2$. In particular, for $pl'>2$ contradicting item~\ref{paper4d}.
	\end{proof}
	
	\begin{proposition}\label{paper4p1}
		There exists a minimizing sequence which converges to a nontrivial weak limit.
	\end{proposition} 
	
	\begin{proof}
		Let $(u_{n},v_{n})_{n}\subset \mathcal{N}$ be the minimizing sequence satisfying \eqref{paper4j21}. By the Lemma~\ref{paper4principal}~\ref{paper4a}, $(u_{n},v_{n})_{n}$ is bounded in $E$. Thus, passing to a subsequence, we may assume that $(u_{n},v_{n})\rightharpoonup(u_{0},v_{0})$ weakly in $E$. Let us define the shift sequence $(\tilde{u}_{n}(x),\tilde{v}_{n}(x))=(u_{n}(x+y_{n}),v_{n}(x+y_{n}))$. Notice that the sequence $(\tilde{u}_{n},\tilde{v}_{n})_{n}$ is also bounded in $E$ which implies that, up to a subsequence, $(\tilde{u}_{n},\tilde{v}_{n})\rightharpoonup(\tilde{u},\tilde{v})$ weakly in $E$. By using assumption \ref{paper4A1}, we can note that the energy functional is invariant by translations of the form $(u,v)\mapsto(u(\cdot-z),v(\cdot-z))$, with $z\in\mathbb{Z}$. Thus, $I(\tilde{u}_{n},\tilde{v}_{n})=I(u_{n},v_{n})$ and $(\tilde{u}_{n},\tilde{v}_{n})_{n}$ is also a minimizing sequence for $I$ on $\mathcal{N}$. Therefore,
		$$
		\lim_{n\rightarrow+\infty}\int_{-R}^{R}(\tilde{u}_{n}^{2}+\tilde{v}_{n}^{2})\;\ud x=\lim_{n\rightarrow+\infty}\int_{y_{n}-R}^{y_{n}+R}(u_{n}^{2}+v_{n}^{2})\;\ud x\geq\beta>0,
		$$
		which implies $(\tilde{u},\tilde{v})\neq(0,0)$. 
	\end{proof}
	
	\noindent For the sake of simplicity, we will keep the notation $(u_{n},v_{n})_{n}$ and $(u_{0},v_{0})$. In order to prove that $(u_{0},v_{0})\in\mathcal{N}$, we will use the following Brezis-Lieb type lemma due to J.M.~do~\'O et al. \cite[Lemma~2.6]{jms}.
	
	\begin{lemma}\label{paper4convergence}
		Let $(u_{n})_{n}\subset H^{1/2}(\mathbb{R})$ be a sequence such that $u_{n}\rightharpoonup u$ weakly in $H^{1/2}(\mathbb{R})$ and $\|u_{n}\|_{1/2}<\rho_{0}$ with $\rho_{0}>0$ small. Then, as $n\rightarrow\infty$, we have
		$$
		\int_{\mathbb{R}}f(u_{n})u_{n}\;\ud x=\int_{\mathbb{R}}f(u_{n}-u)(u_{n}-u)\;\ud x+\int_{\mathbb{R}}f(u)u\;\ud x+o_{n}(1),
		$$
		\vspace{-0,2cm}
		$$
		\int_{\mathbb{R}}F(u_{n})\;\ud x=\int_{\mathbb{R}}F(u_{n}-u)\;\ud x+\int_{\mathbb{R}}F(u)\;\ud x+o_{n}(1).
		$$
	\end{lemma}
	\noindent As consequence of Lemma \ref{paper4convergence}, we have the following lemma:
	
	\begin{lemma}\label{paper4j32}
		If $w_{n}=u_{n}-u_{0}$ and $z_{n}=v_{n}-v_{0}$, then
		\begin{equation}\label{paper4j22}
		\langle I'(u_{0},v_{0}),(u_{0},v_{0})\rangle+\liminf_{n\rightarrow+\infty}\langle I'(w_{n},z_{n}),(w_{n},z_{n})\rangle=0.
		\end{equation}
		Therefore, either $\langle I'(u_{0},v_{0}),(u_{0},v_{0})\rangle\leq0$ or $\liminf_{n\rightarrow+\infty}\langle I'(w_{n},z_{n}),(w_{n},z_{n})\rangle<0$.
	\end{lemma} 
	\begin{proof}
		By easy computations we can deduce that
		$$
		\|u_{n}\|_{E_{1}}^{2}=\|w_{n}\|_{E_{1}}^{2}+\|u_{0}\|_{E_{1}}^{2}+2\left(\int_{\mathbb{R}}(-\Delta)^{1/4}w_{n}(-\Delta)^{1/4}u_{0}\;\ud x+\int_{\mathbb{R}}V_{1}(x)w_{n}u_{0}\;\ud x\right),
		$$
		\vspace{-0,2cm}
		$$
		\|v_{n}\|_{E_{2}}^{2}=\|z_{n}\|_{E_{2}}^{2}+\|v_{0}\|_{E_{2}}^{2}+2\left(\int_{\mathbb{R}}(-\Delta)^{1/4}z_{n}(-\Delta)^{1/4}v_{0}\;\ud x+\int_{\mathbb{R}}V_{2}(x)z_{n}v_{0}\;\ud x\right).
		$$
		Thus, since $(w_{n},z_{n})\rightharpoonup0$ weakly in $E$, we have
		\begin{eqnarray}\label{paper4j24}
		\|(u_{n},v_{n})\|_{E}^{2} & = & \|(w_{n},z_{n})\|_{E}^{2}+\|(u_{0},v_{0})\|_{E}^{2}+2((w_{n},z_{n}),(u_{0},v_{0}))\nonumber\\
		& = & \|(w_{n},z_{n})\|_{E}^{2}+\|(u_{0},v_{0})\|_{E}^{2}+o_{n}(1).
		\end{eqnarray}
		Moreover, we have also that
		$$
		\int_{\mathbb{R}}\lambda(x)w_{n}z_{n}\;\ud x=\int_{\mathbb{R}}\lambda(x)u_{n}v_{n}\;\ud x+\int_{\mathbb{R}}\lambda(x)u_{0}v_{0}\;\ud x-\int_{\mathbb{R}}\lambda(x)u_{n}v_{0}\;\ud x-\int_{\mathbb{R}}\lambda(x)v_{n}u_{0}\;\ud x.
		$$
		By the weak convergence we have the following convergences
		$$
		\int_{\mathbb{R}}\lambda(x)v_{0}u_{n}\;\ud x\rightarrow\int_{\mathbb{R}}\lambda(x)v_{0}u_{0}\;\ud x \quad \mbox{and} \quad \int_{\mathbb{R}}\lambda(x)u_{0}v_{n}\;\ud x\rightarrow\int_{\mathbb{R}}\lambda(x)v_{0}u_{0}\;\ud x,
		$$
		which yields	
		\begin{equation}\label{paper4j25}
		\int_{\mathbb{R}}\lambda(x)w_{n}z_{n}\;\ud x=\int_{\mathbb{R}}\lambda(x)u_{n}v_{n}\;\ud x-\int_{\mathbb{R}}\lambda(x)u_{0}v_{0}\;\ud x+o_{n}(1).
		\end{equation}
		By using Lemma~\ref{paper4convergence}, \eqref{paper4j24}, \eqref{paper4j25} and the fact that $(u_{n},v_{n})_{n}\subset\mathcal{N}$, we conclude that
		$$
		\liminf_{n\rightarrow+\infty}\langle I'(w_{n},z_{n}),(w_{n},z_{n})\rangle=-\langle I'(u_{0},v_{0}),(u_{0},v_{0})\rangle,
		$$
		which completes the proof.
	\end{proof}
	
	\begin{proposition}\label{paper4p2}
		The weak limit $(u_{0},v_{0})$ satisfies $\langle I'(u_{0},v_{0}),(u_{0},v_{0})\rangle=0$.
	\end{proposition}
	\begin{proof}
		We have divided the proof into two steps.
		
		\vspace{0,2cm} 
		\noindent \textit{Step 1.} We first prove that $\langle I'(u_{0},v_{0}),(u_{0},v_{0})\rangle\geq0$.
		\vspace{0,2cm} 
		
		Suppose by contradiction that $\langle I'(u_{0},v_{0}),(u_{0},v_{0})\rangle<0$. Thus, from Lemma \ref{paper4j23}, there exists $t_{0}\in(0,1)$ such that $(t_{0}u_{0},t_{0}v_{0})\in\mathcal{N}$. By using \eqref{paper4j14} and Fatou's lemma, we obtain
		$$
		c_{\mathcal{N}}+o_{n}(1)=\frac{1}{2}\int_{\mathbb{R}}(\phi_{1}(u_{n})+\phi_{2}(v_{n}))\;\ud x\geq \frac{1}{2}\int_{\mathbb{R}}(\phi_{1}(u_{0})+\phi_{2}(v_{0}))\;\ud x+o_{n}(1).
		$$
		Since $t_{0}\in(0,1)$, it follows from \eqref{paper4j31} that
		$$
		\frac{1}{2}\int_{\mathbb{R}}(\phi_{1}(u_{0})+\phi_{2}(v_{0}))\;\ud x+o_{n}(1)>\frac{1}{2}\int_{\mathbb{R}}(\phi_{1}(t_{0}u_{0})+\phi_{2}(t_{0}v_{0}))\;\ud x+o_{n}(1).
		$$
		Combining these estimates and using the fact that $(t_{0}u_{0},t_{0}v_{0})\in\mathcal{N}$, we conclude that
		$$
		c_{\mathcal{N}}+o_{n}(1)>I(t_{0}u_{0},t_{0}v_{0})-\frac{1}{2}\langle I'(t_{0}u_{0},t_{0}v_{0}),(t_{0}u_{0},t_{0}v_{0})\rangle+o_{n}(1)= I(t_{0}u_{0},t_{0}v_{0})+o_{n}(1).
		$$
		Hence, $ I(t_{0}u_{0},t_{0}v_{0})<c_{\mathcal{N}}$, which is a contradiction. Therefore, $\langle I'(u_{0},v_{0}),(u_{0},v_{0})\rangle\geq0$. 
		
		\vspace{0,2cm} 
		\noindent \textit{Step 2.} Now we are going to show that $\langle I'(u_{0},v_{0}),(u_{0},v_{0})\rangle=0$.
		\vspace{0,2cm} 
		
		Suppose by contradiction, that $\langle I'(u_{0},v_{0}),(u_{0},v_{0})\rangle>0$. By Lemma \ref{paper4j32}, we have that
		\begin{equation}\label{paper4j54}
		\liminf_{n\rightarrow+\infty}\langle I'(w_{n},z_{n}),(w_{n},z_{n})\rangle<0.
		\end{equation}
		Thus, passing to a subsequence, we have $\langle I'(w_{n},z_{n}),(w_{n},z_{n})\rangle<0$, for $n\in\mathbb{N}$ sufficiently large. By the Lemma \ref{paper4j23}, there exists a sequence $(t_{n})_{n}\subset(0,1)$ such that $(t_{n}w_{n},t_{n}z_{n})_{n}\subset\mathcal{N}$. Passing to a subsequence, we may assume that $t_{n}\rightarrow t_{0}\in(0,1]$. Arguing by contradiction, we suppose that $t_{0}=1$. Thus, it follows that
		\begin{equation}\label{paper4jj1}
		\|(w_{n},z_{n})\|_{E}^{2}-2\int_{\mathbb{R}}\lambda(x)w_{n}z_{n}\;\ud x=\|(t_{n}w_{n},t_{n}z_{n})\|_{E}^{2}-2\int_{\mathbb{R}}\lambda(x)t_{n}w_{n}t_{n}z_{n}\;\ud x+o_{n}(1).
		\end{equation}
		If we prove the following convergences
		\begin{equation}\label{paper4jj3}
		\int_{\mathbb{R}}f_{1}(w_{n})w_{n}\;\ud x=\int_{\mathbb{R}}f_{1}(t_{n}w_{n})t_{n}w_{n}\;\ud x+o_{n}(1),
		\end{equation}
		\begin{equation}\label{paper4jj4}
		\int_{\mathbb{R}}f_{2}(z_{n})z_{n}\;\ud x=\int_{\mathbb{R}}f_{2}(t_{n}z_{n})t_{n}z_{n}\;\ud x+o_{n}(1),
		\end{equation}
		then combining with \eqref{paper4jj1} and the fact that $(t_{n}w_{n},t_{n}z_{n})_{n}\subset\mathcal{N}$ we conclude that
		$$
		\langle I'(w_{n},z_{n}),(w_{n},z_{n})\rangle=\langle I'(t_{n}w_{n},t_{n}z_{n}),(t_{n}w_{n},t_{n}z_{n})\rangle+o_{n}(1)=o_{n}(1),
		$$
		which contradicts \eqref{paper4j54}. This contradiction implies that $t_{0}\in(0,1)$. It remains to prove \eqref{paper4jj3} and \eqref{paper4jj4}. For this purpose, for each $i=1,2$ we apply the mean value theorem to the function $g_{i}(t)=f_{i}(t)t$. Thus, we get a sequence of functions $(\tau^{i}_{n})_{n}\subset(0,1)$ such that
		\begin{equation}\label{paper4j51}
		f_{1}(w_{n})w_{n}-f_{1}(t_{n}w_{n})t_{n}w_{n}=(f_{1}'(\sigma^{i}_{n})\sigma^{i}_{n}+f_{1}(\sigma^{i}_{n}))w_{n}(1-t_{n}),
		\end{equation}
		\begin{equation}\label{paper4j52}
		f_{2}(z_{n})z_{n}-f_{2}(t_{n}z_{n})t_{n}z_{n}=(f_{2}'(\sigma^{i}_{n})\sigma^{i}_{n}+f_{2}(\sigma^{i}_{n}))z_{n}(1-t_{n}),
		\end{equation}
		where $\sigma_{n}^{1}=w_{n}+\tau^ {1}_{n}w_{n}(t_{n}-1)$ and $\sigma_{n}^{2}=z_{n}+\tau^{2}_{n}z_{n}(t_{n}-1)$. By using Lemma~\ref{paper4principal}~\ref{paper4c}, there exists $\vartheta_{0}>0$ such that $\kappa^{-1}\|(u_{n},v_{n})\|_{E}^{2}\leq\rho_{0}^{2}$, for some $\alpha>\alpha_{0}$, $0<\alpha\rho_{0}^{2}<\omega$ and $\vartheta>\vartheta_{0}$. Since we have
		$$
		\|u_{n}\|_{E_{1}}^{2}=\|w_{n}\|_{E_{1}}^{2}+\|u_{0}\|_{E_{1}}^{2}+o_{n}(1),
		$$
		it follows that $\kappa^{-1}\limsup_{n\rightarrow+\infty}\|w_{n}\|_{E_{1}}^{2}\leq\rho_{0}^{2}$. Thus, up to a subsequence, we get
		$$
		\|\sigma^{1}_{n}\|_{E_{1}}=\|w_{n}+\tau^{1}_{n}w_{n}(t_{n}-1)\|_{E_{1}}= |1-(1-t_{n})\tau^{1}_{n}|\|w_{n}\|_{E_{1}}\leq\kappa\rho_{0},
		$$
		for $n\in\mathbb{N}$ sufficiently large. We claim that
		\begin{equation}\label{paper4j40}
		\sup_{n}\int_{\mathbb{R}}f_{1}(\sigma_{n}^{1})w_{n}\;\ud x<\infty \quad \mbox{and} \quad \sup_{n}\int_{\mathbb{R}}f_{1}'(\sigma_{n}^{1})\sigma_{n}^{1}w_{n}\;\ud x<\infty,
		\end{equation}
		\begin{equation}\label{paper441}
		\sup_{n}\int_{\mathbb{R}}f_{2}(\sigma_{n}^{2})z_{n}\;\ud x<\infty \quad \mbox{and} \quad \sup_{n}\int_{\mathbb{R}}f_{2}'(\sigma_{n}^{2})\sigma_{n}^{2}z_{n}\;\ud x<\infty.
		\end{equation}
		In fact, for $p>2$ it follows from \eqref{paper4ce}, \eqref{paper4j15} and H\"older inequality that 
		$$
		\int_{\mathbb{R}}f_{1}(\sigma_{n}^{1})w_{n}\;\ud x\leq C\|\sigma_{n}^{1}\|_{E_{1}}\|w_{n}\|_{E_{1}}+C\int_{\mathbb{R}}(e^{\alpha(\sigma_{n}^{1})^{2}}-1)|\sigma_{n}^{1}|^{p-1}|w_{n}|\;\ud x.
		$$
		Consider $r>l>1$, sufficiently close to $1$, such that $0<r\alpha\rho_{0}^{2}<\omega$. By using Sobolev embedding, Lemma~\ref{paper4tm}, Lemma~\ref{paper4exp} and H\"older inequality we get
		\begin{eqnarray*}
			\int_{\mathbb{R}}(e^{\alpha(\sigma_{n}^{1})^{2}}-1)|\sigma_{n}^{1}|^{p-1}|w_{n}|\;\ud x & \leq & \left(\int_{\mathbb{R}}(e^{r\alpha(\sigma_{n}^{1})^{2}}-1)\;\ud x\right)^{1/l}\left(\int_{\mathbb{R}}|\sigma_{n}^{1}|^{l'(p-1)}|w_{n}|^{l'}\;\ud x\right)^{1/l'}\\
			& \leq & C\left(\int_{\mathbb{R}}|\sigma_{n}^{1}|^{2l'(p-1)}\;\ud x\right)^{1/2l'}\left(\int_{\mathbb{R}}|w_{n}|^{2l'}\;\ud x\right)^{1/2l'}\\
			& \leq & C\|\sigma_{n}^{1}\|_{E_{1}}^{p-1}\|w_{n}\|_{E_{1}},
		\end{eqnarray*}
		where $1/l+1/l'=1$ and we have used the fact that $2l'(p-1)>2$. Therefore,
		$$
		\int_{\mathbb{R}}f_{1}(\sigma_{n}^{1})w_{n}\;\ud x\leq C\|\sigma_{n}^{1}\|_{E_{1}}\|w_{n}\|_{E_{1}}+C\|\sigma_{n}^{1}\|^{p-1}\|w_{n}\|_{E_{1}}\leq C\rho_{0}^{2}+C\rho_{0}^{p-1}\rho_{0}<\infty.
		$$ 
		By using \eqref{paper4j39} and similar computations we obtain
		$$
		\int_{\mathbb{R}}f_{1}'(\sigma_{n}^{1})\sigma_{n}^{1}w_{n}\;\ud x \leq C\|\sigma_{n}^{1}\|_{E_{1}}\|w_{n}\|_{E_{1}}+C\|\sigma_{n}^{1}\|_{E_{1}}^{p-1}\|w_{n}\|_{E_{1}}<\infty.
		$$
		Analogously we obtain \eqref{paper441} and the claim is proved. 
		
		By using \eqref{paper4j40} and \eqref{paper441} we conclude that
		\begin{equation}\label{paper4j53}
		\sup_{n}\int_{\mathbb{R}}|f_{1}(\sigma_{n}^{1})\sigma_{n}^{1}+f_{1}(\sigma_{n}^{1})||w_{n}|\;\ud x<\infty \quad \mbox{and} \quad \sup_{n}\int_{\mathbb{R}}|f_{2}(\sigma_{n}^{2})\sigma_{n}^{2}+f_{2}(\sigma_{n}^{2})||z_{n}|\;\ud x<\infty.
		\end{equation}
		Finally, combining \eqref{paper4j51}, \eqref{paper4j52}, \eqref{paper4j53} and $t_{n}\rightarrow1$, we get \eqref{paper4jj3} and \eqref{paper4jj4}.
		
		The preceding arguments concluded that, up to a subsequence, $t_{n}\rightarrow t_{0}\in(0,1)$. By a similar argument used in the \textit{Step 1}, we can deduce that
		\begin{equation}\label{paper4j58}
		c_{\mathcal{N}}+o_{n}(1)=\frac{1}{2}\int_{\mathbb{R}}(\phi_{1}(u_{n})+\phi_{2}(v_{n}))\;\ud x\geq\frac{1}{2}\int_{\mathbb{R}}(\phi_{1}(t_{n}u_{n})+\phi_{2}(t_{n}v_{n}))\;\ud x.
		\end{equation} 
		Notice that $t_{n}u_{n}\rightharpoonup t_{0}u_{0}$ and $\kappa^{-1}\|t_{n}u_{n}\|_{E_{1}}^{2}\leq\rho_{0}^{2}$. Thus, by using Lemma \ref{paper4convergence} we have
		\begin{equation}\label{paper4j60}
		\int_{\mathbb{R}}\phi_{1}(t_{n}u_{n})\;\ud x =  \int_{\mathbb{R}}\phi_{1}(t_{n}u_{n}-t_{0}u_{0})\;\ud x+\int_{\mathbb{R}}\phi_{1}(t_{0}u_{0})\;\ud x.
		\end{equation}
		Let us denote $\hat{t}_{n}=t_{n}-t_{0}\rightarrow0$. By the mean value theorem, there exists a sequence of functions $(\gamma_{n})_{n}\subset(0,1)$ such that
		$$
		\phi_{1}(t_{n}u_{n}-t_{0}u_{0})-\phi_{1}(t_{n}w_{n})=\phi_{1}'((1-\gamma_{n})(t_{n}u_{n}-t_{0}u_{0})+\gamma t_{n}w_{n})\hat{t}_{n}u_{0}.
		$$
		Notice that $t_{n}u_{n}-t_{0}u_{0}=t_{n}w_{n}+\hat{t}_{n}u_{0}.$ Thus, it follows that
		\begin{equation}\label{paper4j56}
		\phi_{1}(t_{n}u_{n}-t_{0}u_{0})-\phi_{1}(t_{n}w_{n})=\phi_{1}'(\zeta_{n})\hat{t}_{n}u_{0},
		\end{equation}
		where $\zeta_{n}=(1-\gamma_{n})\hat{t}_{n}u_{0}+t_{n}w_{n}$. Recalling that $\kappa^{-1}\|w_{n}\|_{E_{1}}^{2}\leq\rho_{0}^{2}$ we have
		$$
		\|\zeta_{n}\|_{E_{1}}=\|(1-\gamma_{n})\hat{t}_{n}u_{0}+t_{n}w_{n}\|_{E_{1}}\leq \hat{t}_{n}\|u_{0}\|_{E_{1}}+t_{n}\|w_{n}\|_{E_{1}}\leq\rho_{0},
		$$
		for $n$ sufficiently large. Repeating the same argument used to deduce \eqref{paper4j53}, we get
		\begin{equation}\label{paper4j57} 
		\sup_{n}\int_{\mathbb{R}}|\phi_{1}'(\zeta_{n})||u_{0}|\;\ud x\leq \sup_{n}\int_{\mathbb{R}}|f_{1}'(\zeta_{n})\zeta_{n}+f_{1}(\zeta_{n})||w_{n}|\;\ud x<\infty.
		\end{equation}
		By using \eqref{paper4j56}, \eqref{paper4j57} and the fact that $\hat{t}_{n}\rightarrow0$, we conclude that
		\begin{equation}\label{paper4j59}
		\int_{\mathbb{R}}\phi_{1}(t_{n}u_{n}-t_{0}u_{0})\;\ud x=\int_{\mathbb{R}}\phi_{1}(t_{n}w_{n})\;\ud x+o_{n}(1).
		\end{equation}
		Since $t_{n}v_{n}\rightharpoonup t_{0}v_{0}$ and $\kappa^{-1}\|t_{n}v_{n}\|_{E_{2}}^{2}\leq\rho_{0}^{2}$, we can check analogously that
		\begin{equation}\label{paper4j61}
		\int_{\mathbb{R}}\phi_{2}(t_{n}v_{n})\;\ud x =  \int_{\mathbb{R}}\phi_{2}(t_{n}v_{n}-t_{0}v_{0})\;\ud x+\int_{\mathbb{R}}\phi_{2}(t_{0}v_{0})\;\ud x,
		\end{equation}
		\begin{equation}\label{paper4j62}
		\int_{\mathbb{R}}\phi_{2}(t_{n}v_{n}-t_{0}v_{0})\;\ud x=\int_{\mathbb{R}}\phi_{2}(t_{n}z_{n})\;\ud x+o_{n}(1).
		\end{equation}
		Using \eqref{paper4j58} and the fact that $(u_{n},v_{n})\in\mathcal{N}$ we deduce that
		 $$
		  c_{\mathcal{N}}+o_{n}(1) = I(u_{n},v_{n})-\frac{1}{2}\langle I'(u_{n},v_{n}),(u_{n},v_{n})\rangle \geq \frac{1}{2}\int_{\mathbb{R}}(\phi_{1}(t_{n}u_{n})+\phi_{2}(t_{n}v_{n}))\;\ud x,
		 $$
		which combined with \eqref{paper4j60}, \eqref{paper4j59}, \eqref{paper4j61} and \eqref{paper4j62} implies that
		 $$
		  c_{\mathcal{N}}+o_{n}(1) \geq \frac{1}{2}\int_{\mathbb{R}}(\phi_{1}(t_{n}w_{n})+\phi_{2}(t_{n}z_{n}))\;\ud x+\frac{1}{2}\int_{\mathbb{R}}(\phi_{1}(t_{0}u_{0})+\phi_{2}(t_{0}v_{0}))\;\ud x+o_{n}(1).
		 $$
		Therefore, using the fact that $(t_{n}w_{n},t_{n}z_{n})\in\mathcal{N}$ we conclude that
		 \begin{equation}\label{paper4j123}		 
		  c_{\mathcal{N}}+o_{n}(1) \geq I(t_{n}w_{n},t_{n}z_{n})+\frac{1}{2}\int_{\mathbb{R}}(\phi_{1}(t_{0}u_{0})+\phi_{2}(t_{0}v_{0}))\;\ud x+o_{n}(1).
		 \end{equation}
		Since $(u_{0},v_{0})\neq(0,0)$, it follows from \eqref{paper4j14} that
		$$
		\frac{1}{2}\int_{\mathbb{R}}(\phi_{1}(t_{0}u_{0})+\phi_{2}(t_{0}v_{0}))\;\ud x>0,
		$$
		which jointly with \eqref{paper4j123} implies that $I(t_{n}w_{n},t_{n}z_{n})<c_{\mathcal{N}}$ for $n$ large, contradicting the definition of $c_{\mathcal{N}}$. Therefore, $\langle I'(u_{0},v_{0}),(u_{0},v_{0})\rangle=0$ and the proof is complete.
		

	\end{proof}
	
	\begin{proof}[Proof of Theorem~\ref{paper4A} completed]
		Finally, we will conclude that $(u_{0},v_{0})$ is in fact a ground state solution for System~\eqref{paper4j0}, even though we do not know if $(u_{n},v_{n})$ converges strongly in $E$. By the Propositions \ref{paper4p1} and \ref{paper4p2}, we have that $(u_{0},v_{0})\in\mathcal{N}$. Thus, $c_{\mathcal{N}}\leq I(u_{0},v_{0})$. On the other hand, by using \eqref{paper4j14} and similar arguments as used before, we deduce that
		$$
		c_{\mathcal{N}}+o_{n}(1)=\frac{1}{2}\int_{\mathbb{R}}(\phi_{1}(u_{n})+\phi_{2}(v_{n}))\;\ud x\geq \frac{1}{2}\int_{\mathbb{R}}(\phi_{1}(u_{0})+\phi_{2}(v_{0}))\;\ud x+o_{n}(1)=I(u_{0},v_{0})+o_{n}(1),
		$$
		which implies that $c_{\mathcal{N}}\geq I(u_{0},v_{0})$. Therefore $ I(u_{0},v_{0})=c_{\mathcal{N}}$ and jointly with Remark~\ref{paper4remark} implies that $(u_{0},v_{0})$ is a ground state solution for System~\eqref{paper4j0}. 
		
		In order to get a nonnegative ground state, we note that $I(|u_{0}|,|v_{0}|)\leq I(u_{0},v_{0})$. Moreover, by using Lemma~\ref{paper4j23}, there exists $t_{0}>0$, depending on $(|u_{0}|,|v_{0}|)$, such that $(t_{0}|u_{0}|,t_{0}|v_{0}|)\in\mathcal{N}$. Since $(u_{0},v_{0})\in\mathcal{N}$, we have also from Lemma~\ref{paper4j23} that $\max_{t\geq0}I(tu_{0},tv_{0})=I(u_{0},v_{0})$. Hence,
		$$
		I(t_{0}|u_{0}|,t_{0}|v_{0}|)\leq I(t_{0}u_{0},t_{0}v_{0})\leq \max_{t\geq0}I(tu_{0},tv_{0})=I(u_{0},v_{0})=c_{\mathcal{N}}.
		$$
		Therefore, $(t_{0}|u_{0}|,t_{0}|v_{0}|)\in\mathcal{N}$ is a nonnegative ground state solution for System~\eqref{paper4j0}	which finishes the proof of Theorem \ref{paper4A}.
	\end{proof}
	
	
	\begin{remark}\label{paper4r2}
		Let $\mathcal{K}$ be the set of all ground state solutions for System~\eqref{paper4j0}, that is,
		$$
		\mathcal{K}:=\{(u,v)\in E: (u,v)\in\mathcal{N}, \ I(u,v)=c_{\mathcal{N}} \ \mbox{and} \  I'(u,v)=0\}.
		$$
		We claim that $\mathcal{K}$ is a compact subset of $E$. Indeed, take $(u_{n},v_{n})_{n}\subset\mathcal{K}$ a bounded sequence, thus, up to a subsequence, we may assume $(u_{n},v_{n})\rightharpoonup(u,v)$ weakly in $E$. Proceeding analogously to the proof of Proposition~\ref{paper4p1}, we can conclude that there exists a sequence $(y_{n})_{n}\subset\mathbb{Z}$ and constants $R,\xi>0$ such that
		$$
		\liminf_{n\rightarrow\infty}\int_{B_{R}(y_n)}(u_{n}^{2}+v_{n}^{2})\;\ud x\geq\xi>0.
		$$
		Using the invariance of $I$, we may conclude that $(u,v)\neq0$. Repeating the same arguments used in the proof of Proposition~\ref{paper4p2}, we deduce that $(u,v)\in\mathcal{N}$. As before, we see also that $I(u,v)=c_{\mathcal{N}}$. Thus, using \ref{paper4f3} and Fatou's lemma, we can deduce that
		\begin{eqnarray}
		c_{\mathcal{N}}+o_{n}(1) & = & I(u_{n},v_{n})-\frac{1}{\mu}\langle I'(u_{n},v_{n}),(u_{n},v_{n})\rangle\nonumber\\
		& \geq & I(u,v)-\frac{1}{\mu}\langle I'(u,v),(u,v)\rangle+o_{n}(1)\nonumber\\
		& = & c_{\mathcal{N}}+o_{n}(1)\nonumber.
		\end{eqnarray}
		Thus, $\|(u_{n},v_{n})\|\rightarrow\|(u,v)\|$, which implies that $(u_{n},v_{n})\rightarrow(u,v)$ strongly in $E$. Therefore, $\mathcal{K}$ is a compact set in $E$.
	\end{remark}
	
	
	\section{Proof of Theorem~\ref{paper4B}}\label{paper4t2}
	
	In this section we will be concerned with the existence of ground states for the asymptotically periodic case. We emphasize that the only difference between the potentials $V_{i}(x)$, $\lambda(x)$ and $\tilde{V_{i}}(x)$, $\tilde{\lambda}(x)$ is the periodicity by translations required to $V_{i}(x)$ and $\lambda(x)$. Thus, if $\tilde{V_{i}}(x)$ and $\tilde{\lambda}(x)$ are periodic potentials, we can make use of Theorem~\ref{paper4A} to get a ground state solution for System~\eqref{paper4j00}. Let us suppose that they are not periodic. 
	
	Associated to System~\eqref{paper4j00}, we have the following energy functional 
	$$
	\tilde{I}(u,v)=\frac{1}{2}\left(\|(u,v)\|_{\tilde{E}}^{2}-2\int_{\mathbb{R}}\tilde{\lambda}(x)uv\;\ud x\right)-\int_{\mathbb{R}}\left(F_{1}(u)+F_{2}(v)\right)\;\ud x.
	$$
	
	\noindent The Nehari manifold for System~\eqref{paper4j00} is defined by
	$$
	\tilde{\mathcal{N}}=\{(u,v)\in\tilde{E}\backslash\{(0,0)\}:\langle\tilde{I}'(u,v),(u,v)\rangle=0\},
	$$
	and the ground state energy associated $c_{\tilde{\mathcal{N}}}=\inf_{\tilde{\mathcal{N}}}\tilde{I}(u,v)$. Similarly to Section~\ref{paper4subneh}, for any $(u,v)\in\tilde{\mathcal{N}}$, we can deduce that
	$$   
	\tilde{I}(u,v)\geq\left(\frac{1}{2}-\frac{1}{\mu}\right)(1-\delta)\|(u,v)\|_{\tilde{E}}^{2}\geq \left(\frac{1}{2}-\frac{1}{\mu}\right)(1-\delta)\rho>0.
	$$ 
	Hence, $c_{\tilde{\mathcal{N}}}>0$. The next step is to establish a relation between the levels $c_{\mathcal{N}}$ and $c_{\tilde{\mathcal{N}}}$. 
	
	\begin{lemma}\label{paper4est}
		$c_{\tilde{\mathcal{N}}}<c_{\mathcal{N}}$.
	\end{lemma}
	\begin{proof}
		Let $(u_{0},v_{0})\in \mathcal{N}$ be the nonnegative ground state solution for System~\eqref{paper4j0} obtained by Theorem~\ref{paper4A}. It is easy to see that Lemma~\ref{paper4j23} works for $\tilde{I}$ and $\tilde{\mathcal{N}}$. Thus, there exists a unique $t_{0}>0$, depending only on $(u_{0},v_{0})$, such that $(t_{0}u_{0},t_{0}v_{0})\in\tilde{\mathcal{N}}$. By using \ref{paper4A4} we get
		\begin{align*} 
		\int_{\mathbb{R}}\left[(\tilde{V_{1}}(x)-V_{1}(x))u_{0}^{2}+(\tilde{V_{2}}(x)-V_{2}(x))v_{0}^{2}+(\lambda(x)-\tilde{\lambda}(x))u_{0}v_{0}\right]\ud x<0.
		\end{align*}
		Therefore, $\tilde{I}(t_{0}u_{0},t_{0}v_{0})-I(t_{0}u_{0},t_{0}v_{0})<0$. Since $(u_{0},v_{0})$ is a ground state for System~\eqref{paper4j0} we can use Lemma~\ref{paper4j23} to conclude that
		$$
		c_{\tilde{\mathcal{N}}}\leq\tilde{I}(t_{0}u_{0},t_{0}v_{0})<I(t_{0}u_{0},t_{0}v_{0})\leq\max_{t\geq0}I(tu_{0},tv_{0})=I(u_{0},v_{0})=c_{\mathcal{N}},
		$$   
		and the lemma is proved.
	\end{proof}     
	
	As in the proof of Theorem~\ref{paper4A}, there exists a sequence $(u_{n},v_{n})_{n}\subset\mathcal{N}$ such that
	\begin{equation}\label{paper4jj21}
	\tilde{I}(u_{n},v_{n})\rightarrow c_{\tilde{\mathcal{N}}} \quad \mbox{and} \quad  \tilde{I}'\mid_{\tilde{\mathcal{N}}}(u_{n},v_{n})\rightarrow0.
	\end{equation}    	
	Notice that in the proof of Theorem~\ref{paper4A} the only step we used the periodicity of the potentials was to guarantee that a minimizing sequence converges to a nontrivial limit (see Proposition~\ref{paper4p1}). Thus, Lemma~\ref{paper4principal} remains true for the minimizing sequence obtained above to the asymptotically periodic case. Since $(u_{n},v_{n})_{n}$ is a bounded sequence in $\tilde{E}$, we may assume up to a subsequence that $(u_{n},v_{n})\rightharpoonup(u_{0},v_{0})$ weakly in $\tilde{E}$. The main difficulty is to prove that the weak limit is nontrivial. 
	
	\begin{proposition}\label{paper4p3}
		The weak limit $(u_{0},v_{0})$ of the minimizing sequence $(u_{n},v_{n})_{n}$ is nontrivial.
	\end{proposition}
	\begin{proof}
		Arguing by contradiction, we suppose that $(u_{0},v_{0})=(0,0)$. We may assume that 
		\begin{itemize}
			\item $u_{n}\rightarrow 0$ and $v_{n}\rightarrow 0$ strongly in $L^{p}_{loc}(\mathbb{R})$, for all $2\leq p<\infty$;
			\item $u_{n}(x)\rightarrow 0$ and $v_{n}(x)\rightarrow 0$ almost everywhere in $\mathbb{R}$.
		\end{itemize}       
		It follows from \ref{paper4A4} that for any $\varepsilon>0$ there exists $R>0$ such that
		\begin{equation}\label{paper4jj2}
		|V_{1}(x)-\tilde{V_{1}}(x)|<\varepsilon, \quad |V_{2}(x)-\tilde{V_{2}}(x)|<\varepsilon, \quad |\tilde{\lambda}(x)-\lambda(x)|<\varepsilon, \quad \mbox{for} \hspace{0,2cm} |x|\geq R.
		\end{equation}
		By using \eqref{paper4jj2} and the local convergence there exists $n_{0}\in\mathbb{N}$ such that for $n\geq n_{0}$ we have
		\begin{eqnarray*}
			\left|\int_{\mathbb{R}}(V_{1}(x)-\tilde{V_{1}}(x))u_{n}^{2}\;\ud x\right| & \leq & \int_{\{|x|< R\}}|V_{1}(x)-\tilde{V_{1}}(x)|u_{n}^{2}\;\ud x+\int_{\{|x|\geq R\}}|V_{1}(x)-\tilde{V_{1}}(x)|u_{n}^{2}\;\ud x\\
			& \leq & (\|V_{1}\|_{L^{\infty}_{loc}}+\|\tilde{V_{1}}\|_{L^{\infty}_{loc}})\|u_{n}\|_{L^{2}(B_{R}(0))}^{2}+ C\varepsilon\|u_{n}\|_{\tilde{E_{1}}}^{2}\\
			& \leq & (\|V_{1}\|_{L^{\infty}_{loc}}+\|\tilde{V_{1}}\|_{L^{\infty}_{loc}})\varepsilon + C\varepsilon.
		\end{eqnarray*}                                                                
		Analogously, we can deduce that 
		$$
		\left|\int_{\mathbb{R}}(V_{2}(x)-\tilde{V_{2}}(x))v_{n}^{2}\;\ud x\right|\leq (\|V_{2}\|_{L^{\infty}_{loc}}+\|\tilde{V_{2}}\|_{L^{\infty}_{loc}})\varepsilon + C\varepsilon.
		$$
		We have also from \eqref{paper4jj2} that
		\begin{eqnarray*}
			\left|\int_{\mathbb{R}}(\tilde{\lambda}(x)-\lambda(x))u_{n}v_{n}\;\ud x\right| & \leq & \int_{\{|x|< R\}}|\tilde{\lambda}(x)-\lambda(x)||u_{n}||v_{n}|\;\ud x+\int_{\{|x|\geq R\}}|\tilde{\lambda}(x)-\lambda(x)||u_{n}||v_{n}|\;\ud x\\
			& \leq & (\|\tilde{\lambda}\|_{L^{\infty}_{loc}}+\|\lambda\|_{L^{\infty}_{loc}})\|u_{n}\|_{L^{2}(B_{R}(0))}\|v_{n}\|_{L^{2}(B_{R}(0))}+ C\varepsilon\|u_{n}\|_{\tilde{E_{1}}}\|v_{n}\|_{\tilde{E_{2}}}\\
			& \leq & (\|\tilde{\lambda}\|_{L^{\infty}_{loc}}+\|\lambda\|_{L^{\infty}_{loc}})\varepsilon + C\varepsilon,
		\end{eqnarray*}
		for $n\geq\tilde{n_{0}}$. Therefore, using the estimates obtained above, we can conclude that
		$$
		I(u_{n},v_{n})-\tilde{I}(u_{n},v_{n})=o_{n}(1) \quad \mbox{and} \quad \langle I'(u_{n},v_{n}),(u_{n},v_{n})\rangle-\langle\tilde{I}'(u_{n},v_{n}),(u_{n},v_{n})\rangle=o_{n}(1),
		$$
		which jointly with \eqref{paper4jj21} implies that
		\begin{equation}\label{paper4jj22}
		I(u_{n},v_{n})=c_{\tilde{\mathcal{N}}}+o_{n}(1) \quad \mbox{and} \quad \langle I'(u_{n},v_{n}),(u_{n},v_{n})\rangle=o_{n}(1).
		\end{equation}
		By using Lemma~\ref{paper4j23}, there exists a sequence $(t_{n})_{n}\subset(0,+\infty)$ such that $(t_{n}u_{n},t_{n}v_{n})_{n}\subset\mathcal{N}$.
		
		\vspace{0,3cm}
		
		\noindent\textit{Claim 1.} $\limsup_{n\rightarrow+\infty}t_{n}\leq 1$.
		
		\vspace{0,3cm}
		
		In fact, we suppose by contradiction that there exists $\varepsilon_{0}>0$ such that, up to a subsequence, we have $t_{n}\geq1+\varepsilon_{0}$, for all $n\in\mathbb{N}$. Combining \eqref{paper4jj22} and the fact that $(t_{n}u_{n},t_{n}v_{n})\subset\mathcal{N}$, we can deduce that
		$$
		\int_{\mathbb{R}}\left(\frac{f_{1}(t_{n}u_{n})u_{n}}{t_{n}}-f_{1}(u_{n})u_{n}\right)\;\ud x+\int_{\mathbb{R}}\left(\frac{f_{2}(t_{n}v_{n})v_{n}}{t_{n}}-f_{2}(v_{n})v_{n}\right)\;\ud x=o_{n}(1).
		$$
		By using \eqref{paper4j12} (see \eqref{paper4jj23}) and the fact that $t_{n}\geq1+\varepsilon_{0}$, we have that
		\begin{equation}\label{paper4jj24}
		\int_{\mathbb{R}}\left(\frac{f_{1}((1+\varepsilon_{0})u_{n})u_{n}}{1+\varepsilon_{0}}-f_{1}(u_{n})u_{n}\right)\;\ud x+\int_{\mathbb{R}}\left(\frac{f_{2}((1+\varepsilon_{0})v_{n})v_{n}}{1+\varepsilon_{0}}-f_{2}(v_{n})v_{n}\right)\;\ud x\leq o_{n}(1).
		\end{equation}
		Arguing similar to the proof of Proposition~\ref{paper4p1} we consider the shift sequence $(\tilde{u}_{n}(x),\tilde{v}_{n}(x))=(u_{n}(x+y_{n}),v_{n}(x+y_{n}))$. The sequence $(\tilde{u}_{n}(x),\tilde{v}_{n}(x))$ is bounded in $\tilde{E}$ and, up to a subsequence, $(\tilde{u}_{n}(x),\tilde{v}_{n}(x))\rightharpoonup(\tilde{u},\tilde{v})$. Therefore,
		$$
		\lim_{n\rightarrow+\infty}\int_{-R}^{R}(\tilde{u}_{n}^{2}+\tilde{v}_{n}^{2})\;\ud x=\lim_{n\rightarrow+\infty}\int_{y_{n}-R}^{y_{n}+R}(u_{n}^{2}+v_{n}^{2})\;\ud x\geq\beta>0,
		$$
		which implies $(\tilde{u},\tilde{v})\neq(0,0)$. Thus, by using \eqref{paper4j12}, \eqref{paper4jj24} and Fatou's lemma, we conclude that
		$$
		0<\int_{\mathbb{R}}\left(\frac{f_{1}((1+\varepsilon_{0})\tilde{u})\tilde{u}}{1+\varepsilon_{0}}-f_{1}(\tilde{u})\tilde{u}\right)\;\ud x+\int_{\mathbb{R}}\left(\frac{f_{2}((1+\varepsilon_{0})\tilde{v})\tilde{v}}{1+\varepsilon_{0}}-f_{2}(\tilde{v})\tilde{v}\right)\;\ud x\leq o_{n}(1),
		$$
		which is not possible and finishes the proof of \textit{Claim 1.}
		
		\vspace{0,3cm}
		
		\noindent\textit{Claim 2.} \textit{There exists $n_{0}\in\mathbb{N}$ such that $t_{n}\geq1$, for $n\geq n_{0}$.}
		
		\vspace{0,3cm}
		
		In fact, arguing by contradiction, we suppose that up to a subsequence, $t_{n}<1$. By using \eqref{paper4j31} and the fact that $(t_{n}u_{n},t_{n}v_{n})_{n}\subset\mathcal{N}$ we have
		$$
		c_{\mathcal{N}} \leq \frac{1}{2}\int_{\mathbb{R}}(\phi_{1}(t_{n}u_{n})+\phi_{2}(t_{n}v_{n}))\;\ud x\leq \frac{1}{2}\int_{\mathbb{R}}(\phi_{1}(u_{n})+\phi_{2}(v_{n}))\;\ud x=c_{\tilde{\mathcal{N}}}+o_{n}(1).
		$$
		Therefore, $c_{\mathcal{N}}\leq c_{\tilde{\mathcal{N}}}$ which contradicts Lemma~\ref{paper4est} and finishes the proof of \textit{Claim 2.}
		
		Combining \textit{Claims} $1$ and $2$, we can deduce that
		$$
		\int_{\mathbb{R}}(F_{1}(t_{n}u_{n})-F_{1}(u_{n})+F_{2}(t_{n}v_{n})-F_{2}(v_{n}))\;\ud x=\int_{1}^{t_{n}}\int_{\mathbb{R}}(f_{1}(\tau u_{n})u_{n}+f_{2}(\tau v_{n})v_{n})\;\ud x\ud\tau=o_{n}(1).
		$$
		Moreover, we have that
		$$
		\frac{t_{n}^{2}-1}{2}\left(\|(u_{n},v_{n})\|_{E}^{2}-2\int_{\mathbb{R}}\lambda(x)u_{n}v_{n}\;\ud x\right)=o_{n}(1).
		$$
		These convergences imply that $I(t_{n}u_{n},t_{n}v_{n})-I(u_{n},v_{n})=o_{n}(1)$. Thus, it follows from \eqref{paper4jj22} that
		$$
		c_{\mathcal{N}}\leq I(t_{n}u_{n},t_{n}v_{n})=I(u_{n},v_{n})+o_{n}(1)=c_{\tilde{\mathcal{N}}}+o_{n}(1),
		$$
		which contradicts Lemma~\ref{paper4est}. Therefore, $(u_{0},v_{0})\neq(0,0)$ and the proposition is proved.
	\end{proof}
	
	\begin{proof}[Proof of Theorem~\ref{paper4B} completed]
		We point out that we did not use the periodicity on the potentials $V_{i}(x)$ and $\lambda(x)$ to prove Proposition~\ref{paper4p2}. Thus, since $(u_{0},v_{0})\neq(0,0)$, we can repeat the same proof to conclude that $(u_{0},v_{0})\in\tilde{\mathcal{N}}$. Therefore, we have $c_{\tilde{\mathcal{N}}}\leq I(u_{0},v_{0})$. On the other hand, by using \eqref{paper4j14} and similar arguments as used before, we deduce that
		\begin{eqnarray*}
			c_{\tilde{\mathcal{N}}}+o_{n}(1) & = & \frac{1}{2}\int_{\mathbb{R}}(\phi_{1}(u_{n})+\phi_{2}(v_{n}))\;\ud x\\
			                                 & \geq & \frac{1}{2}\int_{\mathbb{R}}(\phi_{1}(u_{0})+\phi_{2}(v_{0}))\;\ud x+o_{n}(1)\\
			                                 & = & \tilde{I}(u_{0},v_{0})+o_{n}(1),
	    \end{eqnarray*}		
		which implies that $c_{\tilde{\mathcal{N}}}\geq I(u_{0},v_{0})$. Therefore $\tilde{I}(u_{0},v_{0})=c_{\mathcal{N}}$. Repeating the same argument used in the proof of Theorem~\ref{paper4A}, we can deduce that there exists $t_{0}>0$ such that $(t_{0}|u_{0}|,t_{0}|v_{0}|)\in\tilde{\mathcal{N}}$ is a ground state solution for System~\eqref{paper4j00} which finishes the proof of Theorem~\ref{paper4B}.    
	\end{proof}
	
	\begin{remark}
		Let $\tilde{\mathcal{K}}$ be the set of all ground state solutions for System~\eqref{paper4j00}, that is,
		$$
		\tilde{\mathcal{K}}:=\{(u,v)\in \tilde{E}: (u,v)\in\tilde{\mathcal{N}}, \ \tilde{I}(u,v)=c_{\tilde{\mathcal{N}}} \ \mbox{and} \  \tilde{I}'(u,v)=0\}.
		$$
		Using Proposition~\ref{paper441} instead Proposition~\ref{paper4p1}, we can apply a similar argument used in Remark~\ref{paper4p3} , with $I$ replaced by $\tilde{I}$, to conclude that $\tilde{\mathcal{K}}$ is a compact set in $\tilde{E}$.
	\end{remark}
	
	\begin{remark}
		For the local case when System~\eqref{paper4j0} involves the standard Laplacian $-\Delta$ and its defined in the plane $\mathbb{R}^{2}$, the assumption \ref{paper4f4} could be replaced by the following condition
		 \begin{equation}\label{paper4j55}
		  \liminf_{n\rightarrow+\infty} \frac{sf_{i}(s)}{e^{\alpha_{0}^{i}s^{2}}}\geq\beta_{0}>\frac{2e}{\alpha_{0}}.
		 \end{equation}
		In this case, the existence result would not have any dependence of constants. In fact, with the aid of \eqref{paper4j55} we can prove that the ground state energy associated to System~\eqref{paper4j0} is strictly less than $2\pi/\alpha_{0}$. For this purpose, it is consider the following Moser's sequence of functions
		\begin{equation*}
		\omega_{n}(x)=\frac{1}{\sqrt{2\pi}}\left\{
		\begin{array}{clc}
		\sqrt{\log(n)} & \mbox{if} & |x|\leq \displaystyle\frac{r}{n},\\
		\displaystyle\frac{\log\left(r/|x|\right)}{\sqrt{\log(n)}} & \mbox{if} & \displaystyle\frac{r}{n}\leq |x| \leq r,\\
		0 & \mbox{if} & |x|\geq r.
		\end{array}
		\right.
		\end{equation*}
	   For existence results in this direction we refer the readers to \cite{severo,djairo,djr}. 
	   An interesting question is to prove the existence of ground states for \eqref{paper4j0} under a condition of type \eqref{paper4j55}.
	   For that it is crucial to build a Moser's sequence for the fractional case.
	\end{remark}     	
	
	\begin{remark}				
		The main goal of the paper was to prove the existence of ground states for Systems~\eqref{paper4j0} and \eqref{paper4j00}, when the constant $\vartheta$ introduced in \ref{paper4f4} is large enough. In the lemma \ref{paper4nehari}, we proved that the norm of any element that belongs to the Nehari manifold is greater or equal to a positive constant $\rho$, which is strictly less than $\kappa\omega/\alpha_{0}$. However, we note by Lemma~\ref{paper4principal}~\ref{paper4c} that the norm of the minimizing sequence is so small as we want, and it is controlled by the choice of $\vartheta$. Thus, our proof holds for any $\vartheta$ contained in a bounded interval of the real line. Let us consider, for instance,
		$$
		\vartheta^{*}:=\sup\{\vartheta\in\mathbb{R}:\eqref{paper4j0} \ \mbox{has ground states}\}.
		$$		
		Naturally, it arises the following questions: $\vartheta^{*}$ is finite? If $\vartheta^{*}$ is finite, then there exists ground states at $\vartheta=\vartheta^{*}$?
	\end{remark}
	
	\bigskip
	\medskip
	

	\bigskip	
\end{document}